\documentclass[11pt]{article}

\usepackage[T1]{fontenc}
\usepackage{lmodern}
\usepackage[a4paper,margin=2.8cm]{geometry}
\usepackage{amsmath,amssymb,amsthm,mathtools,mathrsfs}
\usepackage[numbers,sort&compress]{natbib}
\usepackage[hidelinks]{hyperref}
\usepackage{microtype}

\hypersetup{
  pdftitle={On the Structure of Higher Zhu Induced Modules in Positive Characteristic},
  pdfauthor={Haihua Huang},
  pdfsubject={Higher Zhu algebras and induced modules}
}

\allowdisplaybreaks

\newtheorem{theorem}{Theorem}[section]
\newtheorem{lemma}[theorem]{Lemma}
\newtheorem{proposition}[theorem]{Proposition}
\newtheorem{corollary}[theorem]{Corollary}
\newcommand{\Ann}{\operatorname{Ann}}
\newcommand{\wt}{\operatorname{wt}}
\newcommand{\End}{\operatorname{End}}
\newcommand{\Hom}{\operatorname{Hom}}
\newcommand{\Res}{\operatorname{Res}}

\title{On the Structure of Higher Zhu Induced Modules in Positive Characteristic}
\author{Haihua Huang\thanks{Email: \texttt{hoiwaa@163.com};
ORCID: 0009-0008-7337-3947.}\\
\small School of General Education, Guangzhou Civil Aviation College,
Guangzhou, China}
\date{}

\begin{document}

\maketitle

\begin{abstract}
Let $F$ be an algebraically closed field of characteristic different from $2$,
and let $V$ be a nonnegatively integer-graded vertex operator algebra over $F$.
We consider higher Zhu algebras obtained by taking a quotient of $V$
by a relation space containing the vectors $D_V^{(h)}a-\binom{-\wt a}{h}a$
for all homogeneous $a\in V$ and integers $h\geq1$, where
$D_V^{(h)}a=a_{-h-1}\mathbf1$. For $n\geq1$ and a nonzero unital module $U$
over the corresponding higher Zhu algebra, we study the quotients $L_n(U)$
of induced modules and determine the kernel of the canonical map
$U\to\Omega_n(L_n(U))/\Omega_{n-1}(L_n(U))$.
This kernel is the largest submodule of $U$ that factors through the
adjacent lower-level algebra. When this submodule is zero, the canonical
map is an isomorphism, and each $\Omega_r$ is the direct sum of the first
$r+1$ homogeneous subspaces. These results are proved using an expansion of
zero modes with integer coefficients and the grading of the standard induced modules.
Moreover, the construction is natural with respect to module homomorphisms,
and $L_n(U)$ is indecomposable in the graded module category whenever $U$
is indecomposable. From nonsplit extensions of modules
for higher Zhu algebras satisfying the stated hypotheses, we construct
nonsplit short exact sequences of weak $V$-modules.
\end{abstract}



\section{Introduction}

Higher Zhu algebras relate graded modules for vertex operator algebras to
modules for associative algebras. Zhu \citep{Zhu96} constructed the algebra
$A(V)$ over the complex numbers and used rational functions in arbitrarily
many variables to construct $V$-modules from $A(V)$-modules.
Dong, Li, and Mason \citep{DLM98} developed higher Zhu algebras and the
associated induction theory over the complex numbers.
Yang \citep{Yang15} gave another construction of the higher-level induction
functor using residue formulas for weak associativity.
Dong and Ren \citep{DR14} developed the representation theory of vertex
operator algebras in positive characteristic.
Ren \citep{Ren17} introduced the corresponding higher Zhu algebras and
induction construction. He \citep{He17} described higher Zhu algebras as
subquotients of universal enveloping algebras over the complex numbers.
Xu \citep{Xu26} established such a description over algebraically closed
fields of characteristic different from $2$. For the elements
$a\circ_{n,t}^{s}b$, we use the parameter range $0\leq s\leq t$ given by
Xu \citep{Xu26}.

In positive characteristic, we consider the vectors
$D_V^{(h)}a=a_{-h-1}\mathbf1$ for all homogeneous $a\in V$ and $h\geq1$.
On every weak module, their zero modes satisfy
$o(D_V^{(h)}a)=\binom{-\wt a}{h}o(a)$.
Griffin \citep[Section~4.2]{Griffin26} defined a quotient space using the
elements $a\circ_{n,t}^{0}b$ and the relations
$D_V^{(h)}a-\binom{-\wt a}{h}a$ for $h\geq1$, and stated its unital
associative algebra structure, the existence of induced modules, and the
isomorphism between their degree-$n$ subspace and $U$.
Under the common hypotheses, his relation space and quotient product agree
with those used here. We denote this algebra by $\widetilde A_n(V)$ and
study the quotient $L_n(U)$ of the induced module.
We use the cited associative algebra structure and zero-mode action on weak
modules, and give a detailed proof for the specific induced quotient defined
by weak associativity relations.
In characteristic zero, the relations involving $D_V^{(h)}a$ follow from
those of the form $(L(-1)+L(0))a$. In positive characteristic, they make the
degree-zero action compatible with all the derivative relations used to
define the Lie algebra $\widehat V$.
In the preliminaries, a finite-dimensional vertex operator algebra
illustrates the difference between the relation space obtained using only
$(L(-1)+L(0))a$ and the one obtained using all
$D_V^{(h)}a-\binom{-\wt a}{h}a$.

For $n\geq1$, write
$I_n=\ker(\widetilde A_n(V)\twoheadrightarrow\widetilde A_{n-1}(V))$.
For a nonzero unital $\widetilde A_n(V)$-module $U$, set
$\Ann_U(I_n)=\{u\in U\mid I_nu=0\}$.
We determine the kernel of the canonical map
$\iota_U:U\to\Omega_n(L_n(U))/\Omega_{n-1}(L_n(U))$:
\[
 \ker\iota_U=\Ann_U(I_n).
\]
The right-hand side is the largest submodule of $U$ that factors through
the adjacent lower-level algebra. We further prove that
\[
 \Ann_U(I_n)=0
 \quad\Longleftrightarrow\quad
 \Omega_r(L_n(U))=\bigoplus_{d=0}^rL_n(U)(d)
 \quad\text{for all }r\geq0.
\]
Under these equivalent conditions, $\iota_U$ is an isomorphism.
These conditions are also equivalent to requiring that every nonzero
submodule of $U$ does not factor through the adjacent lower-level algebra.
Over the complex numbers, Barron, Vander Werf, and Yang \citep{BVWY19}
formulated this submodule condition for the canonical map to be an
isomorphism when $U$ need not be simple.
Using an expansion with integer coefficients for the zero modes of all
$a\circ_{n,t}^{s}b$, we prove the exact kernel formula and these conclusions
about the grading for the quotient algebras and induced modules considered
here.

The zero-mode expansion expresses the action of each relation as a sum of
mode products. The two factors have degrees $k$ and $-k$, respectively, with
$k>n$; for each fixed vector, only finitely many terms act nontrivially.
For the corresponding inclusion over the complex numbers, see
Barron, Vander Werf, and Yang \citep{BVWY19}; related mode computations
appear in Ren \citep{Ren17}. The zero-mode expansion with integer
coefficients used here appears in Griffin
\citep[proof of Proposition~4.13]{Griffin26}.
We record this formula in our mode notation and verify the degrees of the
factors and the meaning of the sum on weak modules.
In the quotient $L_n(U)$, every nonzero homogeneous vector can be sent to a
nonzero vector of degree $n$ by an element of the enveloping algebra.
Together with the zero-mode formula, this property determines the kernel
of the canonical map.

Finally, we verify the naturality of the induction construction with
respect to module homomorphisms. Using the isomorphism between the
degree-$n$ subspace and $U$, together with generation by this subspace, we
obtain a canonical isomorphism of spaces of graded homomorphisms and prove
that $U$ is indecomposable if and only if $L_n(U)$ is indecomposable in the
graded module category.
For a nonsplit extension of simple modules, the induced homomorphism is
injective, and the degree-$n$ subspace of its quotient is isomorphic to the
quotient module in the original extension.
The resulting short exact sequence is nonsplit both in the graded module
category and in the category of weak modules after forgetting the grading.

Throughout the paper, $F$ is an algebraically closed field of characteristic
different from $2$, and $V$ is a nonnegatively integer-graded vertex operator
algebra over $F$. Weights and mode degrees always refer to the given integer
grading. Section~\ref{sec:preliminaries} specifies the relation spaces, modules, and mode actions.
Section~\ref{sec:zero-modes} records the zero-mode formula and the properties of the quotient
algebras. Section~\ref{sec:induced-construction} constructs the induced modules and their quotients and
establishes the properties needed for the subsequent structural results.
Section~\ref{sec:induced-structure} determines the kernel of the canonical map and studies the
grading, homomorphisms, and nonsplit extensions of the induced modules.

\section{Preliminaries}
\label{sec:preliminaries}

Under the standing assumption on the field, we use the definitions of
vertex operator algebras, weak modules, and admissible modules in
Dong and Ren \citep{DR14}.
Write a nonnegatively integer-graded vertex operator algebra as
$V=(V,Y,\mathbf1,\omega)=\bigoplus_{d\geq0}V_d$, where
$\mathbf1\in V_0$ is the vacuum vector, $\omega\in V_2$ is the conformal
vector, and each $V_d$ is a finite-dimensional $F$-vector space. The vertex
operator map is written as $Y(a,z)=\sum_{r\in\mathbb Z}a_rz^{-r-1}$.
For homogeneous $a\in V_d$ and $b\in V_e$, we have
$a_rb\in V_{d+e-r-1}$, and we write $\wt a=d$. The modes of the conformal
vector are denoted by $Y(\omega,z)=\sum_{r\in\mathbb Z}L(r)z^{-r-2}$, so
that $L(r)=\omega_{r+1}$. In particular, $L(0)$ acts on $V_d$ as the image
of the integer $d$ in $F$, and $L(-1)$ is the translation operator.
The notation $\wt a$ always denotes the integer degree $d$.

A weak $V$-module is an $F$-vector space $M$ equipped with a linear map
$Y_M:V\to(\End_F(M))[[z,z^{-1}]]$, written as
$Y_M(a,z)=\sum_{r\in\mathbb Z}a_rz^{-r-1}$. The action satisfies the
truncation condition $a_rw=0$ for fixed $a\in V$ and $w\in M$ and all
sufficiently large $r$, the vacuum property
$Y_M(\mathbf1,z)=\operatorname{Id}_M$, and the module Jacobi identity.
Writing $\delta(z)=\sum_{r\in\mathbb Z}z^r$ for the formal delta function,
the latter identity is
\[
 \begin{aligned}
 &z_0^{-1}\delta\!\left(\frac{z_1-z_2}{z_0}\right)
 Y_M(a,z_1)Y_M(b,z_2)
 -z_0^{-1}\delta\!\left(\frac{z_2-z_1}{-z_0}\right)
 Y_M(b,z_2)Y_M(a,z_1)\\
 &\qquad=
 z_2^{-1}\delta\!\left(\frac{z_1-z_0}{z_2}\right)
 Y_M(Y(a,z_0)b,z_2)
 \end{aligned}
\]
for all $a,b\in V$. The weak module axioms also give
$Y_M(L(-1)a,z)=\frac{d}{dz}Y_M(a,z)$, and we use this derivative property
below.

An admissible $V$-module is a weak module with a nonnegative integer grading
$M=\bigoplus_{d\geq0}M(d)$ such that
$a_rM(d)\subseteq M(d+\wt a-r-1)$. We set $M(k)=0$ for $k<0$.
We allow $M(0)=0$ and always use the given integer grading.
We call $M(d)$ the homogeneous subspace of degree $d$. For a graded
submodule $N\subseteq M$, write $N(d)=N\cap M(d)$.
Graded module homomorphisms preserve the given integer gradings.
All indecomposability statements refer to graded direct sum decompositions.
Conclusions about weak modules after forgetting the grading are specified
explicitly where they occur.

For homogeneous $a\in V$, define its zero mode on $M$ by
$o(a)=a_{\wt a-1}\in\End_F(M)$, and extend $o$ linearly to all of $V$
using the finite homogeneous decomposition of each element.

For a formal Laurent series $f(z)=\sum_jc_jz^j$ with coefficients in a
given vector space, write $[z^q]f(z)=c_q$ and set
$\Res_z f(z)=[z^{-1}]f(z)$. A multivariable residue means the coefficient
of degree $-1$ in each variable in the specified iterated expansion.
For an integer $r$ and a nonnegative integer $h$, we use the generalized
binomial coefficient $\binom rh=[z^h](1+z)^r$, computed first over the
integers and then mapped to $F$.

For a weak $V$-module $M$ and a nonnegative integer $r$, set
\[
 \Omega_r(M)=
 \{w\in M\mid a_jw=0\text{ for all homogeneous $a\in V$ and $j\geq\wt a+r$}\}.
\]
Equivalently, every homogeneous mode of degree strictly less than $-r$
annihilates $\Omega_r(M)$. If $M=\bigoplus_{d\geq0}M(d)$ is admissible,
then $\Omega_r(M)$ is a graded subspace. Indeed, if
$w=\sum_dw_d\in\Omega_r(M)$ and $a_j$ is a homogeneous mode of degree
less than $-r$, the vectors $a_jw_d$ belong to distinct homogeneous
subspaces. The equality $a_jw=0$ and the direct sum grading imply
$a_jw_d=0$ for every $d$. Hence each $w_d$ belongs to $\Omega_r(M)$.

For $h\geq0$, define the operator $D_V^{(h)}a=a_{-h-1}\mathbf1$.
Thus $D_V^{(0)}=\operatorname{id}_V$ and $D_V^{(1)}=L(-1)$, and
$D_V^{(h)}$ has degree $h$.
We use the residue expression of Ren \citep{Ren17}, with the parameter
range $0\leq s\leq t$ used by Xu \citep{Xu26}.
For $n\geq0$, homogeneous $a\in V$, and $b\in V$, set
\begin{equation}
\label{eq:ren-circle}
 a\circ_{n,t}^{s}b
 =\Res_zY(a,z)b\frac{(1+z)^{\wt a+n+s}}{z^{2n+2+t}},
 \qquad 0\leq s\leq t.
\end{equation}
Let $C_n$ be the linear span of these elements, and set
\[
 \mathcal T=\operatorname{span}_F
 \left\{D_V^{(h)}a-\binom{-\wt a}{h}a:
 a\in V\text{ homogeneous},\ h\geq1\right\}.
\]
Define the relation space $\widetilde O_n(V)=C_n+\mathcal T$ and the
quotient vector space $\widetilde A_n(V)=V/\widetilde O_n(V)$.
The binomial expansion $(1+z)^s=\sum_{j=0}^s\binom{s}{j}z^j$ gives
\[
 a\circ_{n,t}^{s}b
 =\sum_{j=0}^{s}\binom{s}{j}a\circ_{n,t-j}^{0}b,
 \qquad 0\leq s\leq t.
\]
Since $t-j\geq0$, the space $C_n$ is also spanned by the elements
$a\circ_{n,t}^{0}b$ for $t\geq0$.
Griffin \citep[Section~4.2]{Griffin26} gives this identity and defines
the relation space using $a\circ_{n,t}^{0}b$ and the above spanning
elements of $\mathcal T$. Under our hypotheses, the two descriptions give
the same relation space.
Define a bilinear product on $V$ by
\begin{equation}
\label{eq:higher-product}
 a*_nb=\sum_{r=0}^{n}(-1)^r\binom{n+r}{n}
 \Res_zY(a,z)b\frac{(1+z)^{\wt a+n}}{z^{n+r+1}},
\end{equation}
first for homogeneous $a$ and then by linear extension using finite
homogeneous decompositions.
The algebra structure on the quotient and its zero-mode action are stated
in Proposition~\ref{prop:modified-zhu-algebra}.

With the above convention on nonnegative parameters, the relation space
used by Ren \citep{Ren17} is
\[
 \begin{aligned}
 O_n(V)={}&\operatorname{span}_F
 \{a\circ_{n,t}^{s}b\mid a\in V\text{ homogeneous},\ b\in V,\ 0\leq s\leq t\}\\
 &+\operatorname{span}_F\{(L(-1)+L(0))a\mid a\in V\}\\
 ={}&C_n+(L(-1)+L(0))V.
 \end{aligned}
\]
Hence $O_n(V)\subseteq\widetilde O_n(V)$.
For homogeneous $a\in V$ and $h\geq1$, the identities
$L(-1)D_V^{(h-1)}a=hD_V^{(h)}a$ and
$D_V^{(h-1)}a\in V_{\wt a+h-1}$ give
\[
 hD_V^{(h)}a+(\wt a+h-1)D_V^{(h-1)}a
 =(L(-1)+L(0))D_V^{(h-1)}a.
\]
In characteristic zero, division by $h$ and induction on $h$ show that
each $D_V^{(h)}a-\binom{-\wt a}{h}a$ is a linear combination of vectors
of the form $(L(-1)+L(0))v$.
In positive characteristic, the two relation spaces can differ.
For example, let $\operatorname{char}F=p>2$, and take
$V=F\mathbf1\oplus F\varepsilon$ with multiplication satisfying
$\varepsilon^2=0$. Define $Y(a,z)b=ab$, $\omega=0$,
$V_0=F\mathbf1$, and $V_p=F\varepsilon$.
The commutative associative multiplication gives the vertex algebra Jacobi
identity. The operators $L(-1)=L(0)=0$ are compatible with the grading
because $p=0$ in $F$. Each homogeneous subspace is finite-dimensional, so
this is a vertex operator algebra under our conventions.
When $n=p$, computing $a\circ_{p,t}^{s}b$ requires the coefficient of
$z^{2p+1+t}$, and $2p+1+t$ is strictly greater than the numerator degree
$\wt a+p+s\leq2p+s$. Thus $C_p=0$.
Since $L(-1)+L(0)=0$, we have $O_p(V)=0$.
On the other hand, $D_V^{(p)}\varepsilon=0$ and
$\binom{-p}{p}=-1$ in $F$, so $\varepsilon\in\mathcal T$.
More explicitly, $(1+z)^{2p}=1+2z^p+z^{2p}$ shows that the terms
$r=0,p$ in \eqref{eq:higher-product} for $\varepsilon*_p\mathbf1$
contribute $2\varepsilon$ and $-2\varepsilon$, respectively, and the
remaining terms vanish. Hence $\varepsilon*_p\mathbf1=0$, whereas
$\mathbf1*_p\varepsilon=\varepsilon$.
Thus adjoining $\mathcal T$ to the relation space affects the identity
element of the quotient and the degree-zero action.
We use tildes to distinguish these quotient spaces.

We use the Lie algebra construction of Ren \citep{Ren17}. Set
$\mathcal L(V)=F[t,t^{-1}]\otimes_FV$. For $i\geq0$, define the operators
$D_t^{(i)}(t^r)=\binom ri t^{r-i}$. Using the previously defined
$D_V^{(i)}$, set
\[
 D^{(i)}=\sum_{j=0}^{i}D_t^{(j)}\otimes D_V^{(i-j)},\qquad
 \mathcal D\mathcal L(V)=\sum_{i>0}D^{(i)}\mathcal L(V),\qquad
 \widehat V=\mathcal L(V)/\mathcal D\mathcal L(V).
\]
Denote the image of $t^r\otimes a$ in $\widehat V$ by $a(r)$.
The Lie bracket is given below, together with the degree of $a(r)$ for
homogeneous $a$:
\[
 [a(p),b(q)]
 =
 \sum_{i\geq0}\binom pi(a_ib)(p+q-i),
 \qquad
 \deg a(r)=\wt a-r-1.
\]
The truncation property makes the sum in the bracket finite for fixed
$a,b$. Thus $\widehat V=\bigoplus_{k\in\mathbb Z}\widehat V_k$ is a
$\mathbb Z$-graded Lie algebra. Its universal enveloping algebra has the
induced grading $U(\widehat V)=\bigoplus_{k\in\mathbb Z}U(\widehat V)_k$.
More explicitly, $U(\widehat V)_k$ is spanned by monomials
$v^{(1)}(r_1)\cdots v^{(\ell)}(r_\ell)$ satisfying
$\sum_{i=1}^{\ell}(\wt v^{(i)}-r_i-1)=k$, where $\ell\geq0$, each
$v^{(i)}\in V$ is homogeneous, and $r_i\in\mathbb Z$. For $\ell=0$,
the monomial is the identity element of the enveloping algebra and has
degree zero.

Let $M$ be a weak $V$-module. For $a,b\in V$ and integers $p,q$,
multiplying the module Jacobi identity by $z_0^pz_2^q$ and taking the
residues in $z_0,z_1,z_2$ gives
\begin{equation}
\label{eq:iterate-mode}
 (a_pb)_q
 =
 \sum_{i\geq0}(-1)^i\binom pi
 \bigl(a_{p-i}b_{q+i}-(-1)^pb_{p+q-i}a_i\bigr).
\end{equation}
For each fixed $w\in M$, truncation gives $b_{q+i}w=0$ and $a_iw=0$
for all sufficiently large $i$. Thus only finitely many terms act
nontrivially on $w$. In \eqref{eq:iterate-mode}, take $b=\mathbf1$ and
$p=-h-1$, where $h\geq0$. The vacuum property gives
$\mathbf1_{-1}=\operatorname{Id}_M$ and $\mathbf1_j=0$ for $j\neq-1$.
Hence the product $a_{-h-1-i}\mathbf1_{q+i}$ can be nonzero only when
$i=-q-1\geq0$, and the product $\mathbf1_{q-h-1-i}a_i$ can be nonzero
only when $i=q-h\geq0$. Using $D_V^{(h)}a=a_{-h-1}\mathbf1$, we obtain
\begin{equation}
\label{eq:divided-mode}
 \begin{aligned}
 (D_V^{(h)}a)_q
 &=\begin{cases}
 (-1)^{-q-1}\binom{-h-1}{-q-1}a_{q-h},&q<0,\\
 0,&0\leq q<h,\\
 (-1)^q\binom{-h-1}{q-h}a_{q-h},&q\geq h,
 \end{cases}\\
 &=(-1)^h\binom qh a_{q-h}.
 \end{aligned}
\end{equation}
The simplification uses the integer binomial identity
$\binom{-h-1}{j}=(-1)^j\binom{h+j}{j}$, together with
$\binom qh=(-1)^h\binom{h-q-1}{h}$ for $q<0$.
Consequently, for $i>0$, the linear assignment
$t^r\otimes a\mapsto a_r$ sends $D^{(i)}(t^r\otimes a)$ to
\[
 \begin{aligned}
 \sum_{j=0}^{i}\binom rj(D_V^{(i-j)}a)_{r-j}
 &=\sum_{j=0}^{i}\binom rj(-1)^{i-j}
   \binom{r-j}{i-j}a_{r-i}\\
 &=\binom ri\sum_{j=0}^{i}(-1)^{i-j}\binom ij a_{r-i}=0.
 \end{aligned}
\]
The binomial identities used here hold over the integers and hence also
over $F$. This assignment therefore vanishes on
$\mathcal D\mathcal L(V)$ and descends to $a(r)\mapsto a_r$.
The mode commutator formula shows that the resulting map preserves the
Lie bracket, giving an action of $\widehat V$ on $M$. By the universal
property of the enveloping algebra, this action extends uniquely to an
algebra homomorphism $\varphi_M:U(\widehat V)\to\End_F(M)$ satisfying
$\varphi_M(a(r))=a_r$. For the corresponding action homomorphism over the
complex numbers, see Barron, Vander Werf, and Yang
\citep[equation~(5)]{BVWY19}.

\section{Zero-mode action and quotient algebras}
\label{sec:zero-modes}

We first specify the subspace notation used in the zero-mode action formula.
For a weak $V$-module $M$ and a linear subspace $X\subseteq V$, write
$o(X)=\{o(c)\mid c\in X\}$ for the set of zero modes of elements of $X$.
Since $o$ is linear, this set is a linear subspace of $\End_F(M)$.
For $Y\subseteq M$, the notation $o(X)Y=0$ means that $o(c)y=0$ for
every $c\in X$ and $y\in Y$. In particular, $o(\widetilde O_m(V))$
denotes the subspace of zero-mode operators on $M$ associated with all
elements of $\widetilde O_m(V)$. For a linear subspace
$X\subseteq U(\widehat V)$, write
$\varphi_M(X)=\{\varphi_M(x)\mid x\in X\}$.
We next specify the notation for composition and summation of operators.
For linear subspaces $X,Y\subseteq U(\widehat V)$, set
\[
 \varphi_M(X)\varphi_M(Y)
 =\operatorname{span}_F
 \{\varphi_M(x)\varphi_M(y)\mid x\in X,\ y\in Y\}
 \subseteq\End_F(M),
\]
where multiplication is composition and $\varphi_M(y)$ acts first.
For a weak $V$-module $M$ and $m\geq0$, consider operators
$T\in\End_F(M)$ that can be written as
\[
 T=\sum_{k>m}T_k,\qquad
 T_k\in
 \varphi_M\bigl(U(\widehat V)_k\bigr)
 \varphi_M\bigl(U(\widehat V)_{-k}\bigr),
\]
with only finitely many nonzero $T_kw$ for each fixed $w\in M$.
We denote the linear subspace consisting of all such $T$ by
$\sum_{k>m}\varphi_M(U(\widehat V)_k)\varphi_M(U(\widehat V)_{-k})$.
Every subsequent use of this summation notation requires that only finitely
many $T_kw$ be nonzero for each fixed $w\in M$.
The following inclusion follows from the zero-mode expansion. We record
the degrees of both mode factors and check that only finitely many terms
act nontrivially on each fixed vector, for later use.

\begin{theorem}
\label{thm:zero-mode-action}
\citep{Griffin26}
Let $M$ be a weak $V$-module. For every $m\geq0$, the following inclusion
holds in $\End_F(M)$:
\begin{equation}
\label{eq:MF}
 o\bigl(\widetilde O_m(V)\bigr)
 \subseteq
 \sum_{k>m}
 \varphi_M\bigl(U(\widehat V)_k\bigr)
 \varphi_M\bigl(U(\widehat V)_{-k}\bigr).
\end{equation}
\end{theorem}

\begin{proof}
The element $a\circ_{m,t}^{s}b$ is linear in $b$, and $b$ has a finite
homogeneous decomposition. It therefore suffices to treat homogeneous $b$.
Fix homogeneous $a,b\in V$ and write
$\alpha=\wt a$, $\beta=\wt b$, $r=\alpha+m+s$, and $N=2m+2+t$,
where $0\leq s\leq t$. Equation~\eqref{eq:ren-circle} and the definition
of the zero mode give
\begin{equation}
\label{eq:circle-zero-mode-expansion}
 \begin{aligned}
 o\bigl(a\circ_{m,t}^{s}b\bigr)
 &=\sum_{j=0}^{r}\binom rj o(a_{j-N}b)\\
 &=\sum_{j=0}^{r}\binom rj
 (a_{j-N}b)_{\alpha+\beta+N-j-2}.
 \end{aligned}
\end{equation}
Here each nonzero $a_{j-N}b$ has weight $\alpha+\beta+N-j-1$.
The mode form of the Jacobi identity in Griffin
\citep[proof of Proposition~4.13]{Griffin26} gives the following formula.
Its weak-module version is obtained by multiplying the weak module
Jacobi identity by $z_0^{-N}z_1^{r}z_2^{\beta+m+t-s}$ and taking residues
in all three variables. All binomial coefficients are evaluated in
$\mathbb Z$ and then mapped to $F$:
\begin{equation}
\label{eq:zero-mode-formula}
 \begin{aligned}
 o\bigl(a\circ_{m,t}^{s}b\bigr)
 ={}&
 \sum_{\ell\geq0}(-1)^\ell\binom{-N}{\ell}
 a_{\alpha-m+s-t-2-\ell}b_{\beta+m-s+t+\ell}\\
 &-
 \sum_{\ell\geq0}(-1)^{N+\ell}\binom{-N}{\ell}
 b_{\beta-m-s-2-\ell}a_{\alpha+m+s+\ell}.
 \end{aligned}
\end{equation}

By the definition of mode degree,
\[
 \begin{aligned}
 \deg a_{\alpha-m+s-t-2-\ell}&=m+1+t-s+\ell,\\
 \deg b_{\beta+m-s+t+\ell}&=-(m+1+t-s+\ell),\\
 \deg b_{\beta-m-s-2-\ell}&=m+1+s+\ell,\\
 \deg a_{\alpha+m+s+\ell}&=-(m+1+s+\ell).
 \end{aligned}
\]
Since $0\leq s\leq t$ and $\ell\geq0$, both positive degrees are at least
$m+1$. Thus the two products in \eqref{eq:zero-mode-formula} satisfy
\[
 \begin{aligned}
 a_{\alpha-m+s-t-2-\ell}b_{\beta+m-s+t+\ell}
 &\in\varphi_M(U(\widehat V)_k)\varphi_M(U(\widehat V)_{-k}),
 &&k=m+1+t-s+\ell>m,\\
 b_{\beta-m-s-2-\ell}a_{\alpha+m+s+\ell}
 &\in\varphi_M(U(\widehat V)_k)\varphi_M(U(\widehat V)_{-k}),
 &&k=m+1+s+\ell>m.
 \end{aligned}
\]

Fix $w\in M$. For all sufficiently large $\ell$, weak module truncation
gives $b_{\beta+m-s+t+\ell}w=0$ and $a_{\alpha+m+s+\ell}w=0$.
Thus only finitely many terms of \eqref{eq:zero-mode-formula} act
nontrivially on each fixed vector.

Finally, for homogeneous $a\in V_\alpha$, equation~\eqref{eq:divided-mode} gives
\begin{equation}
\label{eq:translation-zero-mode}
 o(D_V^{(h)}a)=(-1)^h\binom{\alpha+h-1}{h}o(a)
 =\binom{-\alpha}{h}o(a),\qquad o(\mathcal T)=0.
\end{equation}
Every element of $\widetilde O_m(V)$ is a finite linear combination of
elements $a\circ_{m,t}^{s}b$ and $D_V^{(h)}a-\binom{-\wt a}{h}a$.
Combining finitely many of the sums above still gives only finitely many
nonzero terms on each fixed vector. This proves \eqref{eq:MF}.
\end{proof}

We use the unital associative algebra structure on $\widetilde A_n(V)$
from Griffin \citep[Section~4.2]{Griffin26} and the zero-mode action on
$\Omega_n(M)$ from \citep[Definition/\allowbreak Lemma~4.9]{Griffin26}.
The relation space and multiplication there agree with the definition
following \eqref{eq:ren-circle} and with \eqref{eq:higher-product}, respectively.
The natural surjection between consecutive levels is verified directly
from these two definitions after the proposition.

\begin{proposition}
\label{prop:modified-zhu-algebra}
\citep{Griffin26}
For each $n\geq0$, the product $*_n$ induces an associative algebra
structure on $\widetilde A_n(V)$ with identity
$\mathbf1+\widetilde O_n(V)$. For every weak $V$-module $M$, the space
$\Omega_n(M)$ is a unital $\widetilde A_n(V)$-module under
$a+\widetilde O_n(V)\mapsto o(a)$. For $n\geq1$, the identity map
induces a surjective algebra homomorphism
$\widetilde A_n(V)\to\widetilde A_{n-1}(V)$.
\end{proposition}

Fix $n\geq1$. For the relation spaces at consecutive levels,
equation~\eqref{eq:ren-circle} gives
\[
 a\circ_{n,t}^{s}b=a\circ_{n-1,t+2}^{s+1}b,
 \qquad 0\leq s\leq t,
\]
where $0\leq s+1\leq t+2$. Hence $C_n\subseteq C_{n-1}$ and
$\widetilde O_n(V)\subseteq\widetilde O_{n-1}(V)$.
To verify compatibility with multiplication, combine equal powers using
Pascal's identity to obtain
\[
 \begin{aligned}
 &(1+z)\sum_{r=0}^{n}(-1)^r\binom{n+r}{n}z^{-n-r-1}
 -\sum_{r=0}^{n-1}(-1)^r\binom{n-1+r}{n-1}z^{-n-r}\\
 &\qquad=(-1)^n\left(\binom{2n-1}{n}z^{-2n}
                   +\binom{2n}{n}z^{-2n-1}\right).
 \end{aligned}
\]
For homogeneous $a\in V$ and $b\in V$, multiplying by
$(1+z)^{\wt a+n-1}Y(a,z)b$ and taking the residue gives
\[
 a*_nb-a*_{n-1}b
 =(-1)^n\left(\binom{2n-1}{n}a\circ_{n-1,0}^{0}b
             +\binom{2n}{n}a\circ_{n-1,1}^{0}b\right)
 \in C_{n-1}.
\]
Extending over the finite homogeneous decomposition of $a$ shows that
this relation holds for all $a,b\in V$. Thus
$a+\widetilde O_n(V)\mapsto a+\widetilde O_{n-1}(V)$ is a surjective
map preserving multiplication and the identity.

\section{Construction of induced modules}
\label{sec:induced-construction}

Griffin \citep[Section~4.2]{Griffin26} states the existence of the induced
modules below and the isomorphism between their degree-$n$ subspaces and
the given algebra modules. In this section, we start with an induced Lie
algebra module and define a specific quotient by the weak associativity
relations. We prove that the submodule generated by these relations has
zero intersection with the given $U$, and verify admissibility and the
degree-$n$ isomorphism for the quotient. We then construct the quotient
$L_n(U)$ and establish the properties used in the next section to study
the kernel of the canonical map and the graded structure.

To define the action of $\widehat V_0$, we first compute
$a*_nb-b*_na$ modulo $\mathcal T$. Fix $n\geq0$, and let $a,b\in V$
be homogeneous, with $\alpha=\wt a$ and $\beta=\wt b$.
The congruences below mean that the difference lies in $\mathcal T$;
congruences of formal series are interpreted coefficientwise.
The skew-symmetry formula
$Y(b,z)a=\sum_{h\geq0}z^hD_V^{(h)}Y(a,-z)b$
and the weight $\alpha+\beta-k-1$ of $a_kb$ give
\[
 Y(b,z)a\equiv
 (1+z)^{-\alpha-\beta}
 Y\left(a,-\frac{z}{1+z}\right)b
 \pmod{\mathcal T((z))}.
\]
For each fixed coefficient, truncation ensures that the sums involved
are finite. The substitution $z=-x/(1+x)$ in the residue gives
\[
 b*_na\equiv(-1)^n\sum_{r=0}^n\binom{n+r}{n}
 \Res_x\frac{(1+x)^{\alpha+r-1}}{x^{n+r+1}}Y(a,x)b
 \pmod{\mathcal T}.
\]
The linear coefficient of this substitution is $-1$. A termwise
calculation for integer powers $z^k$ shows that the residue after the
change of variable is $1$ for $k=-1$ and zero otherwise. Thus the
residue change-of-variable formula holds over the present base field.

To compare the two products, we use the following rational function identity:
\[
 \sum_{r=0}^n\binom{n+r}{n}
 \frac{(-1)^r(1+x)^{n+1}-(-1)^n(1+x)^r}{x^{n+r+1}}=1.
\]
This identity can be proved in $\mathbb Z[y]$. Set
$F_n(y)=\sum_{r=0}^n\binom{n+r}{n}y^r$.
The binomial expansion in $\mathbb Z[[y]]$ gives
\[
 F_n(y)\equiv(1-y)^{-n-1}\pmod{y^{n+1}},\qquad
 (1-y)^{n+1}F_n(y)\equiv1\pmod{y^{n+1}}.
\]
Hence the polynomial
$(1-y)^{n+1}F_n(y)+y^{n+1}F_n(1-y)-1$
is divisible by $y^{n+1}$. This polynomial is invariant under
$y\mapsto1-y$, so it is also divisible by $(1-y)^{n+1}$.
Since
\[
 1=(y+(1-y))^{2n+1}
 \in(y^{n+1})+((1-y)^{n+1})\subseteq\mathbb Z[y],
\]
the sum of these two principal ideals is the whole ring, and their
intersection therefore equals their product. Thus the polynomial is
divisible by $y^{n+1}(1-y)^{n+1}$; since its degree is at most $2n+1$,
it is zero. Substituting $y=-1/x$ gives the required identity.
Factoring out $(1+x)^{\alpha-1}$ from the difference of the two
products and applying this identity yields
\begin{equation}
\label{eq:modified-commutator}
 a*_nb-b*_na\equiv
 \Res_z(1+z)^{\alpha-1}Y(a,z)b\pmod{\mathcal T}.
\end{equation}
This formula extends bilinearly using the finite homogeneous
decompositions of $a$ and $b$.

Equip $\widetilde A_n(V)$ with the commutator Lie bracket. We verify
that the assignment from the degree-zero subspace of $\mathcal L(V)$
to $\widetilde A_n(V)$ is compatible with the derivative relations.
The degree-zero derivative relations in $\mathcal L(V)$ are spanned by
$D^{(h)}(t^{\alpha+h-1}\otimes a)$. The assignment
$t^{\wt b-1}\otimes b\mapsto b+\widetilde O_n(V)$ sends these relations to
\[
 \sum_{j=0}^h\binom{\alpha+h-1}{j}
 \binom{-\alpha}{h-j}(a+\widetilde O_n(V))
 =\binom{h-1}{h}(a+\widetilde O_n(V))=0.
\]
Together with \eqref{eq:modified-commutator}, this gives a surjective
Lie algebra homomorphism
\begin{equation}
\label{eq:degree-zero-map}
 \widehat V_0\longrightarrow\widetilde A_n(V),\qquad
 a(\wt a-1)\longmapsto a+\widetilde O_n(V).
\end{equation}

Fix $n\geq1$. Using \eqref{eq:degree-zero-map}, we construct an
admissible module from any unital $\widetilde A_n(V)$-module $U$.
Set $P_n=\widehat V_0\oplus\bigoplus_{d<-n}\widehat V_d$.
The bracket respects degrees, so $P_n$ is a Lie subalgebra.
The degree-zero part acts on $U$ through \eqref{eq:degree-zero-map},
and the remaining part acts trivially. Place $U$ in degree $n$ and define
\[
 \overline M_n(U)=U(\widehat V)\otimes_{U(P_n)}U,\qquad
 Y_{\overline M_n(U)}(a,z)=\sum_{r\in\mathbb Z}a(r)z^{-r-1}.
\]
The PBW theorem ensures that $u\mapsto1\otimes u$ is injective, and
we write $u$ for its image. On the induced module $\overline M_n(U)$
and the quotient modules defined below, we denote the formal series
defined by these mode actions by $Y(a,z)$. The construction by induction
followed by quotienting by the weak associativity relations is due to
Dong, Li, and Mason \citep{DLM98} and Ren \citep{Ren17}.
We prove below that, for a module $U$ over the present quotient algebra,
the canonical map gives $U\cong M_n(U)(n)$.

If $N$ is a $U(\widehat V)$-module and $X\subseteq U(\widehat V)$ and
$B\subseteq N$ are linear subspaces, write
$XB=\operatorname{span}_F\{xb\mid x\in X,\ b\in B\}$. The equality
$XB=0$ means that every $x\in X$ annihilates every $b\in B$.
We use this notation for the submodule generated by the weak
associativity relations.

For homogeneous $a,b$, write $\alpha=\wt a$ and $\beta=\wt b$, and set
$\ell=\alpha+n$. For $s,t\in\mathbb Z$ and $u\in U$, put
\begin{equation}
\label{eq:association-coefficient}
 \begin{aligned}
 w_{s,t}(a,b;u)={}&
 \sum_{i\geq0}(-1)^i\binom si
 a(s+\ell-i)b(t+i)u\\
 &-\sum_{j=0}^{\ell}\binom{\ell}{j}
 (a_{s+j}b)(t+\ell-j)u.
 \end{aligned}
\end{equation}
Since $b(t+i)u=0$ for $t+i\geq\beta+n$, this expression is a finite
sum. Let $W$ be the linear span of these vectors, and set
$R=U(\widehat V) W$ and $M_n(U)=\overline M_n(U)/R$.
Expression~\eqref{eq:association-coefficient} is the residue of
\[
 (z_0+z_2)^{\ell}Y(a,z_0+z_2)Y(b,z_2)u
 -(z_2+z_0)^{\ell}Y(Y(a,z_0)b,z_2)u
\]
after multiplication by $z_0^sz_2^t$; the substitution in the product
of vertex operators is expanded in powers of $z_2/z_0$.
Each relation vector has degree $\beta-s-t-2$, so $R$ is graded.

In constructing $V$-modules from $A(V)$-modules over the complex numbers,
Zhu \citep{Zhu96} uses rational functions in arbitrarily many variables
and extracts coefficients corresponding to mode actions.
Write $A=\widetilde A_n(V)$, and let $\rho:V\to A$ be the quotient map.
We use $\rho(\mathcal T)=0$ and the vanishing of the elements
$a\circ_{n,t}^{s}b$ in the quotient to construct rational functions
in arbitrarily many variables. These functions allow us to treat the
action of any universal enveloping algebra element on the left of a
relation vector. Linear maps acting on vector-valued formal series and
rational functions are extended coefficientwise.
For algebraically independent indeterminates $x_1,\ldots,x_m$, let
$F[x_i^{\pm1},(x_i-x_j)^{-1}:i<j]$ denote the $F$-algebra generated by
$x_i$, $x_i^{-1}$ ($1\leq i\leq m$), and
$(x_i-x_j)^{-1}$ ($1\leq i<j\leq m$).
It is a subring of the rational function field $F(x_1,\ldots,x_m)$,
and its elements have the form
\[
 \frac{P(x_1,\ldots,x_m)}
 {\prod_{i=1}^m x_i^{e_i}\prod_{1\leq i<j\leq m}(x_i-x_j)^{N_{ij}}},
 \qquad P\in F[x_1,\ldots,x_m],\quad e_i,N_{ij}\in\mathbb Z_{\geq0}.
\]
The notation $\iota_{x_1,\ldots,x_m}$ denotes the expansion map from
this rational function ring to the iterated Laurent series field
$F((x_1))\cdots((x_m))$. It fixes the elements of $F$ and each $x_i$,
is specified by
\[
 \iota_{x_1,\ldots,x_m}(x_i-x_j)^{-1}
 =\sum_{r\geq0}x_i^{-r-1}x_j^r,
 \qquad 1\leq i<j\leq m,
\]
and extends by addition and multiplication. Thus, for $i<j$, the
expansion is in nonnegative integer powers of $x_j/x_i$.
This map is injective. For rational functions with coefficients in
$A$, we use the same expansion, leaving the coefficients unchanged.

\begin{lemma}
\label{lem:multivariable-functions}
Let $m\geq1$, let $a_1,\ldots,a_m,c\in V$ be homogeneous, and write
$\alpha_i=\wt a_i$. The formal series
\[
 \rho\left(\prod_{i=1}^m(1+x_i)^{\alpha_i+n}
 Y(a_1,x_1)\cdots Y(a_m,x_m)c\right)
\]
is the $\iota_{x_1,\ldots,x_m}$ expansion of a rational function
$G_{a_1,\ldots,a_m;c}\in
 A\otimes_FF[x_i^{\pm1},(x_i-x_j)^{-1}:i<j]$.
In the subscript of $G_{a_1,\ldots,a_m;c}$, the entries
$a_1,\ldots,a_m$ correspond, in order, to
$Y(a_1,x_1),\ldots,Y(a_m,x_m)$, and the entry $c$ after the semicolon
is the vector on which this product of vertex operators acts.
For each $i$, regard the other variables as parameters and take the
Laurent expansion of this function in $x_i^{-1}$. Every exponent of
$x_i$ in this expansion is at most $2n$.
Interchanging the positions of $a_m$ and $c$ gives the following
transformation formula:
\begin{equation}
\label{eq:multivariable-change-base}
 \begin{aligned}
 &G_{a_1,\ldots,a_m;c}(x_1,\ldots,x_m)\\
 &\quad=(1+x_m)^{(m+1)n}
 G_{a_1,\ldots,a_{m-1},c;a_m}
 \left(\frac{x_1-x_m}{1+x_m},\ldots,
 \frac{x_{m-1}-x_m}{1+x_m},-\frac{x_m}{1+x_m}\right).
 \end{aligned}
\end{equation}
For fixed $a_1,\ldots,a_m,c$, abbreviate this function to $G$.
If $X_1,\ldots,X_m$ are algebraically independent indeterminates, then
at $t=0$ we have
\begin{equation}
\label{eq:multivariable-total-degree}
 G(t^{-1}X_1,\ldots,t^{-1}X_m)
 \in t^{-(m+1)n}(A\otimes_FF(X_1,\ldots,X_m))[[t]].
\end{equation}
\end{lemma}

\begin{proof}
Choose $e_i\geq0$ such that $x_i^{e_i}Y(a_i,x_i)c\in V[[x_i]]$.
By locality of vertex operators, we can choose nonnegative integers
$N_{ij}=N_{ji}$ such that
$(x_i-x_j)^{N_{ij}}[Y(a_i,x_i),Y(a_j,x_j)]=0$. Set
\[
 \begin{aligned}
 Q&=\prod_{k=1}^m x_k^{e_k}
       \prod_{k<\ell}(x_k-x_\ell)^{N_{k\ell}},\\
 P&=Q\prod_{k=1}^m(1+x_k)^{\alpha_k+n}
       \rho\bigl(Y(a_1,x_1)\cdots Y(a_m,x_m)c\bigr).
 \end{aligned}
\]
Fix $i$, and list the remaining indices in increasing order as
$j_1,\ldots,j_{m-1}$. Locality gives the following equalities of
formal series:
\[
 \begin{aligned}
 P
 &=Q\prod_{k=1}^m(1+x_k)^{\alpha_k+n}
   \rho\bigl(Y(a_{j_1},x_{j_1})\cdots
     Y(a_{j_{m-1}},x_{j_{m-1}})Y(a_i,x_i)c\bigr)\\
 &=Q\prod_{k=1}^m(1+x_k)^{\alpha_k+n}
   \rho\bigl(Y(a_i,x_i)Y(a_{j_1},x_{j_1})\cdots
     Y(a_{j_{m-1}},x_{j_{m-1}})c\bigr).
 \end{aligned}
\]
In the expression with $Y(a_i,x_i)$ on the far right,
$x_i^{e_i}Y(a_i,x_i)c\in V[[x_i]]$, and the other scalar factors
are polynomials in $x_i$. Hence $[x_i^r]P=0$ for $r<0$.
In the expression with $Y(a_i,x_i)$ on the far left, each coefficient
$v\in V$ of the series to its right satisfies
\[
 \begin{aligned}
 [x_i^r]\rho\bigl((1+x_i)^{\alpha_i+n}Y(a_i,x_i)v\bigr)
 &=\rho\left(\Res_{x_i}
     \frac{(1+x_i)^{\alpha_i+n}}{x_i^{r+1}}Y(a_i,x_i)v\right)\\
 &=\rho\bigl(a_i\circ_{n,r-2n-1}^{0}v\bigr)
 =0,\qquad r\geq2n+1.
 \end{aligned}
\]
Together with $\deg_{x_i}Q=e_i+\sum_{j\ne i}N_{ij}$, this gives
\[
 [x_i^r]P=0
 \quad\text{if}\quad
 r<0\ \text{or}\ r>e_i+2n+\sum_{j\ne i}N_{ij}.
\]
These bounds hold for every $i$, so $P\in A[x_1,\ldots,x_m]$.
Define the rational function $G=P/Q$. By the definition of $P$, in the
specified space of iterated Laurent series we have
\[
 \begin{aligned}
 \iota_{x_1,\ldots,x_m}G
 &=P\,\iota_{x_1,\ldots,x_m}(Q^{-1})\\
 &=\prod_{i=1}^m(1+x_i)^{\alpha_i+n}
   \rho\bigl(Y(a_1,x_1)\cdots Y(a_m,x_m)c\bigr).
 \end{aligned}
\]
The product of vertex operators on the right is interpreted through
the mode expansions and is the formal series in the statement of the
lemma. The displayed equality expresses this series as the
$\iota_{x_1,\ldots,x_m}$ expansion of the rational function $G$.
As a polynomial in $x_i$, $Q$ has leading coefficient a nonzero rational
function of the remaining variables. Therefore the largest exponent of $x_i$
in the Laurent expansion of $P/Q$ in $x_i^{-1}$ is at most
\[
 \left(e_i+2n+\sum_{j\ne i}N_{ij}\right)
 -\left(e_i+\sum_{j\ne i}N_{ij}\right)=2n.
\]
The coefficients of $G$ lie in the subspace spanned by the finitely
many coefficients of $P$.

Let $D(y)=\sum_{h\geq0}y^hD_V^{(h)}$. For homogeneous $v$, the
equality $\rho(\mathcal T)=0$ gives
\[
 \rho(D(y)v)
 =\sum_{h\geq0}y^h\rho(D_V^{(h)}v)
 =\sum_{h\geq0}\binom{-\wt v}{h}y^h\rho(v)
 =(1+y)^{-\wt v}\rho(v).
\]
Skew-symmetry and the translation formula give
\[
 \begin{aligned}
 &Y(a_1,x_1)\cdots Y(a_m,x_m)c\\
 &=Y(a_1,x_1)\cdots Y(a_{m-1},x_{m-1})
 D(x_m)Y(c,-x_m)a_m\\
 &=D(x_m)
 Y(a_1,x_1-x_m)\cdots Y(a_{m-1},x_{m-1}-x_m)
 Y(c,-x_m)a_m.
 \end{aligned}
\]
Set $\gamma=\wt c$. For independent variables $z_1,\ldots,z_m$, write
\[
 \begin{aligned}
 &Y(a_1,z_1)\cdots Y(a_{m-1},z_{m-1})Y(c,z_m)a_m\\
 &\qquad=\sum_{r_1,\ldots,r_m}v_{r_1,\ldots,r_m}z_1^{r_1}\cdots z_m^{r_m},
 \qquad
 v_{r_1,\ldots,r_m}\in V_{\sum_i\alpha_i+\gamma+\sum_i r_i}.
 \end{aligned}
\]
Applying the preceding translation relation coefficientwise yields
\[
 \begin{aligned}
 &\rho\bigl(D(y)Y(a_1,z_1)\cdots
       Y(a_{m-1},z_{m-1})Y(c,z_m)a_m\bigr)\\
 &=\sum_{r_1,\ldots,r_m}
   (1+y)^{-\sum_i\alpha_i-\gamma-\sum_i r_i}
   \rho(v_{r_1,\ldots,r_m})z_1^{r_1}\cdots z_m^{r_m}\\
 &=(1+y)^{-\sum_i\alpha_i-\gamma}
   \rho\left(Y\left(a_1,\frac{z_1}{1+y}\right)\cdots
       Y\left(a_{m-1},\frac{z_{m-1}}{1+y}\right)
       Y\left(c,\frac{z_m}{1+y}\right)a_m\right).
 \end{aligned}
\]
Set $y=x_m$, $z_i=x_i-x_m$ ($i<m$), and $z_m=-x_m$, expanding
each $(x_i-x_m)^{r_i}$ in powers of $x_m/x_i$.
Choose $K\geq0$ such that $Y(c,z_m)a_m\in z_m^{-K}V[[z_m]]$.
If $h$ is the index of a term of $D(x_m)$ and $j_i\geq0$ is the
expansion index for $(x_i-x_m)^{r_i}$, then a contribution to a fixed
coefficient of $x_m^q$ satisfies
\[
 h+\sum_{i<m}j_i+r_m=q,\qquad r_m\geq-K,
 \qquad 0\leq h+\sum_{i<m}j_i\leq q+K.
\]
After fixing the powers of the other variables as well, only finitely
many terms contribute. Thus these substitutions and coefficientwise
operations are well defined.

Put $y_i=(x_i-x_m)/(1+x_m)$ ($i<m$) and
$y_m=-x_m/(1+x_m)$. Then
\[
 1+y_i=\frac{1+x_i}{1+x_m}\ (i<m),\qquad
 1+y_m=\frac1{1+x_m}.
\]
Multiplying the preceding equality by $\prod_i(1+x_i)^{\alpha_i+n}$ gives
\[
 \begin{aligned}
 &G_{a_1,\ldots,a_m;c}(x_1,\ldots,x_m)\\
 &=\frac{\prod_i(1+x_i)^{\alpha_i+n}}
 {(1+x_m)^{\sum_i\alpha_i+\gamma}
   \prod_{i<m}(1+y_i)^{\alpha_i+n}(1+y_m)^{\gamma+n}}
 G_{a_1,\ldots,a_{m-1},c;a_m}(y_1,\ldots,y_m)\\
 &=(1+x_m)^{(m+1)n}
 G_{a_1,\ldots,a_{m-1},c;a_m}(y_1,\ldots,y_m).
 \end{aligned}
\]
Here we first compare the corresponding expansions. Injectivity of
the expansion map then gives this equality of rational functions,
which is \eqref{eq:multivariable-change-base}.

Finally, substitute $x_i=t^{-1}X_i$, and denote the resulting
coordinates by $y_i(t)$. We have
\[
 \begin{aligned}
 y_i(t)&=\frac{X_i-X_m}{t+X_m},
 &y_i(0)&=\frac{X_i-X_m}{X_m}\neq0 &&(i<m),\\
 y_m(t)&=-\frac{X_m}{t+X_m},
 &y_m(0)&=-1\neq0,\\
 y_i(t)-y_j(t)&=\frac{X_i-X_j}{t+X_m}
 &&&& (i<j<m),\\
 y_i(t)-y_m(t)&=\frac{X_i}{t+X_m}
 &&&& (i<m).
 \end{aligned}
\]
All coordinates and coordinate differences are nonzero at $t=0$, so
\[
 G_{a_1,\ldots,a_{m-1},c;a_m}(y_1(t),\ldots,y_m(t))
 \in (A\otimes_FF(X_1,\ldots,X_m))[[t]].
\]
Consequently,
\[
 \begin{aligned}
 &G(t^{-1}X_1,\ldots,t^{-1}X_m)\\
 &=t^{-(m+1)n}(t+X_m)^{(m+1)n}
   G_{a_1,\ldots,a_{m-1},c;a_m}(y_1(t),\ldots,y_m(t))\\
 &\in t^{-(m+1)n}(A\otimes_FF(X_1,\ldots,X_m))[[t]],
 \end{aligned}
\]
which proves \eqref{eq:multivariable-total-degree}.
\end{proof}

We now use the coefficients of these rational functions to construct a linear
map and prove that the canonical map $U\to M_n(U)$ is injective.
The operators required below are defined by
$\partial_z^{(h)}z^r=\binom rh z^{r-h}$ for $h\geq0$ and $r\in\mathbb Z$,
and extended linearly to formal series by applying them term by term.
We use this convention throughout the rest of the paper.

\begin{lemma}
\label{lem:induction-coefficient-map}
Let $n\geq1$, and let $U$ be a unital $\widetilde A_n(V)$-module.
There exists a graded linear map $\tau:\overline M_n(U)\to U$ such that
\[
 \tau|_U=\operatorname{id}_U,\qquad \tau(R)=0,
\]
where $U$ is concentrated in degree $n$. In particular, $R\cap U=0$, and
the canonical map
\[
 U\longrightarrow M_n(U)(n),\qquad u\longmapsto1\otimes u+R
\]
is injective.
\end{lemma}

\begin{proof}
For $q\geq2$ and homogeneous $a_i\in V_{\alpha_i}$, define
\begin{equation}
\label{eq:rational-coefficients}
 H_q(a_1,\ldots,a_q;z_1,\ldots,z_q)=
 z_q^{(q-1)n-\alpha_q}\prod_{i<q}z_i^{-\alpha_i-n}
 G_{a_1,\ldots,a_{q-1};a_q}
 \left(\frac{z_1-z_q}{z_q},\ldots,\frac{z_{q-1}-z_q}{z_q}\right).
\end{equation}
Also set $H_1(a;z)=z^{-\wt a}\rho(a)$ and
$H_0=\rho(\mathbf1)$, and extend linearly using the homogeneous
decomposition of each vector. In each $H_q$ below, the vectors
$a_1,\ldots,a_q$ precede the variables $z_1,\ldots,z_q$, with the two
groups of arguments separated by a semicolon.
By definition, for an independent variable $t$,
$H_q(a_1,\ldots,a_q;tz_1,\ldots,tz_q)
=t^{-\sum_i\alpha_i}H_q(a_1,\ldots,a_q;z_1,\ldots,z_q)$.
Locality allows simultaneous permutations of $a_1,\ldots,a_{q-1}$ and
their corresponding variables. Equation~\eqref{eq:multivariable-change-base}
also allows us to interchange $a_{q-1},a_q$ and $z_{q-1},z_q$
simultaneously. Thus $H_q$ is invariant under simultaneous permutations
of the vectors and their corresponding variables.

Set $z_q=t$. Equation~\eqref{eq:multivariable-total-degree} gives
\[
 G(z_1/t-1,\ldots,z_{q-1}/t-1)
 \in t^{-qn}(A\otimes_FF(z_1,\ldots,z_{q-1}))[[t]].
\]
Here replacing $X_i$ by $z_i-t$ preserves the total degree bound:
we may write $G$ as a polynomial divided by a product of powers of the
coordinates and their differences, and after simultaneous scaling the
leading term of each denominator is nonzero in the independent
variables $z_i$. Together with \eqref{eq:rational-coefficients} and
symmetry, this shows that the Laurent expansion of $H_q$ at each
$z_i=0$ contains no powers below $-\alpha_i-n$. Here the expansion at
$z_i=0$ is taken over the field of rational functions in the remaining
variables. Hence, for suitable $N_{ij}\geq0$,
\begin{equation}
\label{eq:rational-pole-bound}
 H_q=
 \frac{P_q(z_1,\ldots,z_q)}
 {\prod_i z_i^{\alpha_i+n}\prod_{i<j}(z_i-z_j)^{N_{ij}}},
 \qquad P_q\in A[z_1,\ldots,z_q].
\end{equation}
The numerator may be chosen homogeneous.

Substituting $z_i=z_q+\xi_i$ and expanding in the order
$\xi_1,\ldots,\xi_{q-1}$ transforms
\eqref{eq:rational-coefficients} into
\[
 z_q^{-\sum_i\alpha_i}
 \rho(Y(a_1,\xi_1/z_q)\cdots Y(a_{q-1},\xi_{q-1}/z_q)a_q).
\]
Use symmetry to place $b$ last and $a$ next to last, with corresponding
variables $y$ and $y+x$, respectively. Expand $Y(a,x/y)b$ term by term,
and then return the vectors and their corresponding variables to their
original positions. Regard $y$ and the remaining variables as parameters
and take the Laurent expansion of the left-hand side below at $x=0$.
In this expansion, $(y+x)^{-k}$ and $(z_i-y-x)^{-k}$ are expanded in
nonnegative powers of $x/y$ and $x/(z_i-y)$, respectively, for $k\geq1$.
Since $a_rb=0$ for all sufficiently large $r$, the right-hand side has
only finitely many negative powers of $x$. With this convention,
\begin{equation}
\label{eq:rational-iterate}
 H_q(\ldots,a,b,\ldots;\ldots,y+x,y,\ldots)
 =\sum_{r\in\mathbb Z}x^{-r-1}
 H_{q-1}(\ldots,a_rb,\ldots;\ldots,y,\ldots).
\end{equation}
Using the weight $\wt a+\wt b-r-1$ of each $a_rb$, the power of $y$
computed from \eqref{eq:rational-coefficients} agrees with that on the
right-hand side. Applying $\partial_{z_i}^{(h)}$ for $i<q$, and then
using symmetry for $i=q$, gives
\[
 H_q(a_1,\ldots,D_V^{(h)}a_i,\ldots,a_q;z_1,\ldots,z_q)
 =\partial_{z_i}^{(h)}H_q(a_1,\ldots,a_q;z_1,\ldots,z_q).
\]
For $q=1$, this is the identity
$\rho(D_V^{(h)}a)=\binom{-\wt a}{h}\rho(a)$.
If some $a_i=\mathbf1$, the vacuum property, the formula
$Y(v,z)\mathbf1=\sum_{h\geq0}z^hD_V^{(h)}v$, and the definition of
$H_q$ allow us to remove that vector and its corresponding variable
from the arguments.

To construct a linear map from the enveloping algebra to $A$, first
define an $F$-linear map $f$ on the tensor algebra of $\mathcal L(V)$.
Until the quotient relations have been checked, let $a(r)$ denote
$t^r\otimes a\in\mathcal L(V)$ and let juxtaposition denote multiplication
in the tensor algebra. For $q\geq1$, set
\[
 f(a_1(r_1)\cdots a_q(r_q))
 =\Res_{z_1,\ldots,z_q}z_1^{r_1}\cdots z_q^{r_q}
 \iota_{z_1,\ldots,z_q}H_q,
 \qquad f(1)=\rho(\mathbf1).
\]
The function $H_q$ is multilinear in the vectors, and each coefficient
extracted is a finite linear combination of elements of $A$. Thus this
rule is well-defined on the tensor algebra.
For a formal Laurent series $B(z)$, a coefficientwise calculation gives
\[
 \begin{aligned}
 \Res_z z^r\partial_z^{(h)}B(z)
 &=\binom{h-r-1}{h}[z^{h-r-1}]B(z)\\
 &=(-1)^h\binom rh\Res_z z^{r-h}B(z).
 \end{aligned}
\]
Consequently, with arbitrary products of modes fixed on the left and right,
\[
 f(\cdots(D_V^{(h)}a)(r)\cdots)
 =(-1)^h\binom rh f(\cdots a(r-h)\cdots).
\]
For the derivative relations in $\mathcal L(V)$, this gives
\[
 \begin{aligned}
 &f\left(\cdots\sum_{j=0}^h\binom rj
             (D_V^{(h-j)}a)(r-j)\cdots\right)\\
 &=\binom rh\sum_{j=0}^h(-1)^{h-j}\binom hj
       f(\cdots a(r-h)\cdots)=0,\qquad h>0.
 \end{aligned}
\]
When adjacent modes are interchanged, the difference between the two
iterated expansions is determined by the poles where the corresponding
variables coincide. Using the finite principal part in
\eqref{eq:rational-iterate} and the integer identity
\[
 (\iota_{z,w}-\iota_{w,z})(z-w)^{-i-1}
 =\partial_w^{(i)}\left(z^{-1}\delta(w/z)\right),
\]
we obtain
\[
 \begin{aligned}
 &f(\cdots a(r)b(s)\cdots)
 -f(\cdots b(s)a(r)\cdots)\\
 &\qquad=\sum_{i\geq0}\binom ri
 f(\cdots(a_i b)(r+s-i)\cdots).
 \end{aligned}
\]
The linear map therefore descends to $f:U(\widehat V)\to A$.
From now on, $a(r)$ again denotes the image of $t^r\otimes a$ in the
enveloping algebra. Homogeneity further gives
$f(U(\widehat V)_d)=0$ for $d\ne0$ and
$f(a(\alpha-1))=\rho(a)$.

Let $\pi_U:A\to\End_F(U)$ be the given action. To define
$\tau(x\otimes u)=\pi_U(f(x))u$, we now verify the defining relations
of the balanced tensor product. The bound in
\eqref{eq:rational-pole-bound} on the order of the pole at zero in the
variable corresponding to the rightmost mode $b(s)$ gives
$f(xb(s))=0$ for $s\geq\wt b+n$.
For the residue calculation below, write
$K_n(z)=\sum_{r=0}^n(-1)^r\binom{n+r}{n}z^{-n-r-1}$.
For a degree-zero mode, first consider the two-point function.
Let $\wt v=h$ and $\wt b=\beta$, and set
$S(x)=(1+x)^{h+n}\rho(Y(v,x)b)$.
Lemma~\ref{lem:multivariable-functions} gives
$S\in A[x,x^{-1}]$ and $\deg S\leq2n$. Setting $\lambda=y/w$, we have
\[
 \begin{aligned}
 \Res_w w^{\beta-1}\iota_{y,w}H_2(v,b;y,w)
 &=y^{-h}[\lambda^0]_{\infty}\lambda^{-n}S(\lambda-1)\\
 &=y^{-h}[x^{-1}]_{\infty}(1+x)^{-n-1}S(x)\\
 &=y^{-h}\rho(v*_nb).
 \end{aligned}
\]
The subscript $\infty$ denotes expansion in negative powers of the variable.
The last step uses
\[
 \begin{aligned}
 (1+x)^{-n-1}-K_n(x)&\in x^{-2n-2}F[[x^{-1}]],\\
 \bigl((1+x)^{-n-1}-K_n(x)\bigr)S(x)&\in x^{-2}A[[x^{-1}]],\\
 [x^{-1}]_\infty(1+x)^{-n-1}S(x)
 &=[x^{-1}]_\infty K_n(x)S(x)=\rho(v*_nb).
 \end{aligned}
\]

For any $q\geq1$ homogeneous vectors $a_i\in V_{\alpha_i}$, write the
corresponding variables as $y+\xi_1,\ldots,y+\xi_{q-1},y$.
Set $\xi=(\xi_1,\ldots,\xi_{q-1})$. For
$j=(j_1,\ldots,j_{q-1})\in\mathbb Z^{q-1}$, write
$\xi^{j}=\prod_{i=1}^{q-1}\xi_i^{j_i}$, and expand
\[
 Y(a_1,\xi_1)\cdots Y(a_{q-1},\xi_{q-1})a_q
 =\sum_{j}v_{j}\xi^{j}.
\]
For $q=1$, the right-hand side consists only of the coefficient $a_1$.
Each nonzero coefficient has weight $\sum_i\alpha_i+\sum_i j_i$.
Successively applying \eqref{eq:rational-iterate} to $H_q$ and to
$H_{q+1}$ obtained by adding $b$ and its variable $w$ at the end gives,
respectively,
\[
 \sum_{j}\xi^{j}
 y^{-\wt v_{j}}\rho(v_{j}),
 \qquad
 \sum_{j}\xi^{j}
 H_2(v_{j},b;y,w).
\]
Every residue in $w$ below is taken by first expanding in Laurent series
at $w=0$ over the field of rational functions in the remaining
coordinates, and then extracting the coefficient of $w^{-1}$.
In the above expansion in $\xi$, the two-point calculation and the
algebra action on $U$ give
\[
 \begin{aligned}
 \Res_w w^{\beta-1}\pi_U H_{q+1}
 &=\sum_j\xi^j y^{-\wt v_j}\pi_U(\rho(v_j*_nb))\\
 &=\sum_j\xi^j y^{-\wt v_j}
       \pi_U(\rho(v_j))\pi_U(\rho(b))\\
 &=(\pi_U H_q)\pi_U(\rho(b)).
 \end{aligned}
\]
Here each denominator factor $y+\xi_i-w$ has constant term $y$ at
$w=\xi=0$, each $\xi_i-\xi_j$ is independent of $w$, and the powers of
$w$ have a uniform lower bound. Thus each fixed coefficient involves
only finitely many terms, and taking coefficients in $w$ is compatible
with the above expansion in $\xi$. Since this expansion map is injective,
the resulting identity is an identity of rational functions.
Taking the coefficient corresponding to the product of modes on the left
then gives
$\pi_U(f(xb(\beta-1)))=\pi_U(f(x))\pi_U(\rho(b))$.
For $x=1$, the same identity follows from $f(b(\beta-1))=\rho(b)$ and
$f(1)=\rho(\mathbf1)$. It therefore holds for every $x\in U(\widehat V)$.
Applying the degree-zero and negative-degree relations above to successive
factors of a product in $U(P_n)$ gives
$\pi_U(f(xp))u=\pi_U(f(x))(pu)$ for $x\in U(\widehat V)$,
$p\in U(P_n)$, and $u\in U$.
Hence $\tau(x\otimes u)=\pi_U(f(x))u$ is well-defined on the balanced
tensor product and satisfies
\[
 \tau(1\otimes u)=\pi_U(\rho(\mathbf1))u=u,\qquad
 \tau(\overline M_n(U)(d))=0\quad(d\neq n).
\]

To prove $\tau(R)=0$, take homogeneous $a,b\in V$ and $u\in U$, and
write $\alpha=\wt a$ and $\beta=\wt b$. For $q\geq0$, take an arbitrary
homogeneous product of modes $X=c_1(r_1)\cdots c_q(r_q)$ of degree $d$,
with $X=1$ when $q=0$.
In $H_{q+2}$, extract the coefficient corresponding to $X$, retaining
the power of the difference of the last two variables in the denominator.
By \eqref{eq:rational-pole-bound}, place the variables corresponding to
$X$ before the last two variables and take the iterated expansion in this
order. Denominators involving differences between these two groups of
variables contribute only nonnegative powers of the last two variables.
After multiplication by
$z_1^{\alpha+n}z_2^{\beta+n}(z_1-z_2)^N$ to clear the denominators,
the resulting expression belongs to $A[[z_1,z_2]]$ and contains only
terms of total degree $2n+N-d$.
If $2n+N-d\geq0$, only finitely many pairs of nonnegative integers have
this sum; if $2n+N-d<0$, the expression is zero.
The numerator is therefore a polynomial, and the resulting function has
the form
\[
 \frac{P_X(z_1,z_2)}
 {z_1^{\alpha+n}z_2^{\beta+n}(z_1-z_2)^N},
 \qquad P_X\in A[z_1,z_2].
\]
Before extracting the coefficient corresponding to $X$, the expansion
after clearing these denominators also contains only nonnegative powers
of $z_1,z_2$. After substituting $z_1=z_0+z_2$, for fixed $i,j\geq0$
the coefficient of $z_0^iz_2^j$ involves only the finitely many terms
$z_1^hz_2^k$ with $h+k=i+j$ and $h\geq i$, each contributing the factor
$\binom hi$. Thus extracting the coefficient corresponding to $X$ is
compatible with this substitution. Multiplication by
$z_2^{-\beta-n}z_0^{-N}$ only shifts the exponents, so this compatibility
is preserved.
Multiplying by $z_1^{\alpha+n}$ and substituting $z_1=z_0+z_2$ gives
the Laurent polynomial
$z_2^{-\beta-n}z_0^{-N}P_X(z_0+z_2,z_2)$.
The iterated expansion in \eqref{eq:rational-iterate}, multiplied by
$(z_2+z_0)^{\alpha+n}$, gives the same Laurent polynomial.
Since both expansions equal this Laurent polynomial,
\[
 \begin{aligned}
 \tau(Xw_{s,t}(a,b;u))
 &=\pi_U\left(\Res_{z_0,z_2}z_0^sz_2^t
 \left(
 \frac{P_X(z_0+z_2,z_2)}{z_2^{\beta+n}z_0^N}
 -\frac{P_X(z_0+z_2,z_2)}{z_2^{\beta+n}z_0^N}
 \right)\right)u\\
 &=0.
 \end{aligned}
\]
Every element of the enveloping algebra is a finite linear combination
of products of modes. Hence $\tau(R)=\tau(U(\widehat V)W)=0$.
If $u\in R\cap U$, then $u=\tau(u)=0$. The kernel of the canonical
map above is $R\cap U=0$, so the map is injective.
\end{proof}

For an $A_d(V)$-module $U$, Griffin \citep[Section~4.2]{Griffin26}
states the existence of an admissible $V$-module
\[
 M_d(U)=\bigoplus_{i\geq0}M_d(U)_i,\qquad M_d(U)_d\cong U,
\]
where his $A_d(V)$ corresponds to our $\widetilde A_d(V)$.
We give a detailed proof of these conclusions for the specific quotient
$M_n(U)=\overline M_n(U)/R$ defined above.
Lemma~\ref{lem:induction-coefficient-map} proves that the canonical map
$U\to M_n(U)$ is injective. We next prove $M_n(U)(d)=0$ for $d<0$,
$M_n(U)(n)=U$, and the module Jacobi identity on the quotient.

\begin{theorem}
\label{thm:modified-induction}
For every unital $\widetilde A_n(V)$-module $U$, the quotient $M_n(U)$
is an admissible $V$-module. The canonical map $u\mapsto1\otimes u+R$
gives an isomorphism of $\widetilde A_n(V)$-modules
$U\cong M_n(U)(n)$, and $M_n(U)$ is generated by its degree-$n$ subspace.
\end{theorem}

\begin{proof}
We first prove truncation for the formal series on $\overline M_n(U)$.
Take homogeneous $x_1,\ldots,x_s\in\widehat V$ and $u\in U$, write
$v=x_1\cdots x_su$, and set
$B=n+\sum_{i=1}^s\max(\deg x_i,0)$.
Induction on the number $s$ of factors gives $\widehat V_dv=0$ for
$d<-B$. The case $s=0$ follows from the action of $P_n$.
For the induction step, write $v=x_1v'$ and set
$B'=n+\sum_{i=2}^s\max(\deg x_i,0)$.
For $y\in\widehat V_d$, the induction hypothesis gives
\[
 \begin{aligned}
 d<-B&=-B'-\max(\deg x_1,0)
 \quad\Longrightarrow\quad
 d<-B',\quad d+\deg x_1<-B',\\
 yv&=x_1yv'+[y,x_1]v'=0.
 \end{aligned}
\]
Here $B=B'+\max(\deg x_1,0)$.
Thus, for homogeneous $a\in V_\alpha$, if $r\geq\alpha+B$, then
$\deg a(r)=\alpha-r-1<-B$, and hence $a(r)v=0$.
In $\widehat V$, the derivative relations and induction on $h\geq1$ give
\[
 \begin{aligned}
 (D_V^{(h)}a)(r)
 &=-\sum_{j=1}^h\binom rj(D_V^{(h-j)}a)(r-j)\\
 &=-\sum_{j=1}^h\binom rj(-1)^{h-j}
       \binom{r-j}{h-j}a(r-h)\\
 &=-\binom rh\sum_{j=1}^h(-1)^{h-j}\binom hj a(r-h)
 =(-1)^h\binom rh a(r-h).
 \end{aligned}
\]
Since $D_V^{(h)}\mathbf1=0$, we have
$\binom{r+h}{h}\mathbf1(r)=0$.
For $r\leq-2$, take $h=-r-1$; for $r\geq0$ in characteristic $p>2$,
take $h=p^e>r$. The corresponding coefficients are $(-1)^h$ and $1$,
respectively. The latter equality follows by comparing the coefficients
of $z^{p^e}$ in $(1+z)^{r+p^e}=(1+z)^r(1+z^{p^e})$.
In characteristic zero, take $h=1$ in the latter case.
Thus $\mathbf1(r)=0$ for $r\neq-1$.
The centrality of $\mathbf1(-1)$ and its action on $U$ give
$\mathbf1(-1)xu=x\mathbf1(-1)u=xu$ for $x\in U(\widehat V)$ and
$u\in U$. Hence
$Y_{\overline M_n(U)}(\mathbf1,z)=\mathrm{id}_{\overline M_n(U)}$.

Lemma~\ref{lem:induction-coefficient-map} gives $R\cap U=0$.
We next determine a spanning set for the degree-$n$ subspace.
For homogeneous $a,b$, write $\alpha=\wt a$ and $\beta=\wt b$, and set
$\ell=\alpha+n$. Taking the parameters $r-\ell,s$ in
\eqref{eq:association-coefficient} gives, in the quotient,
\[
 a(r)b(s)u
 =\sum_{j=0}^{\ell}\binom{\ell}{j}
 (a_{r-\ell+j}b)(s+\ell-j)u
 -\sum_{i\geq1}(-1)^i\binom{r-\ell}{i}a(r-i)b(s+i)u.
\]
Fix homogeneous $b$. We induct on $D=\max(\beta+n-s,0)$ to prove,
simultaneously for all homogeneous $a$, integers $r$, and $u\in U$,
that $a(r)b(s)u$ is a finite linear combination of vectors $c(t)u$
with $c$ homogeneous.
For $D=0$, we have $b(s)u=0$. For $D>0$, each term in the second sum
above has $i\geq1$ and satisfies
$\max(\beta+n-s-i,0)<D$, so the induction hypothesis applies.
This sum is finite by truncation of $b(s+i)u$, and the first sum already
consists of single-mode actions.
Now induct on the number of modes, replacing the action of the rightmost
two modes by the finite sum of single-mode actions just obtained.
This reduces the number of modes in each term by one.
All relations preserve the integer degree. Therefore
\[
 M_n(U)(d)=\operatorname{span}_F
 \{a(\wt a+n-d-1)u:a\in V\text{ homogeneous},\ u\in U\}.
\]
For $d<0$, we have $d-n<-n$, so the action of $P_n$ gives
$a(\wt a+n-d-1)U\subseteq\widehat V_{d-n}U=0$.
Thus $M_n(U)(d)=0$.
Together with $R\cap U=0$ and the unitality of the module $U$, this gives
\[
 M_n(U)(n)
 =\operatorname{span}_F\{o(a)u:a\in V\text{ homogeneous},\ u\in U\}
 =\widetilde A_n(V)U=U.
\]
If a product of modes $Xu$ has degree $d<0$, then
$Xu+R\in M_n(U)(d)=0$, so $Xu\in R$.
Hence $yXu\in U(\widehat V)R=R$ for every $y\in U(\widehat V)$.

It remains to verify the module Jacobi identity on the quotient.
The Lie bracket on $\widehat V$ gives
\[
 [Y(a,z_1),Y(b,z_2)]
 =\sum_{i\geq0}Y(a_i b,z_2)
 \partial_{z_2}^{(i)}\left(z_1^{-1}\delta(z_2/z_1)\right).
\]
The right-hand side is a finite sum, so these formal series satisfy locality.
For $u\in U$, take $\ell=\alpha+n$ and $k=\beta+n$, and choose a
nonnegative integer $N$ such that
$(z_1-z_2)^N[Y(a,z_1),Y(b,z_2)]=0$.
Truncation and locality give
\[
 \begin{aligned}
 P_u(z_1,z_2)
 &:=z_1^\ell z_2^k(z_1-z_2)^NY(a,z_1)Y(b,z_2)u\\
 &=z_1^\ell z_2^k(z_1-z_2)^NY(b,z_2)Y(a,z_1)u
 \in M_n(U)[[z_1,z_2]].
 \end{aligned}
\]
The weak associativity relations defining $W$ further give
\[
 z_0^N(z_2+z_0)^\ell z_2^kY(Y(a,z_0)b,z_2)u
 =P_u(z_2+z_0,z_2).
\]
Expand $P_u$ into monomials. For each monomial $z_1^iz_2^j$, multiply
the formal delta identity
\[
 z_0^{-1}\delta\left(\frac{z_1-z_2}{z_0}\right)
 -z_0^{-1}\delta\left(\frac{z_2-z_1}{-z_0}\right)
 =z_2^{-1}\delta\left(\frac{z_1-z_0}{z_2}\right),
\]
by $z_0^{-N}z_1^{i-\ell}z_2^{j-k}$, where the three differences are
expanded in $z_2/z_1$, $z_1/z_2$, and $z_0/z_1$, respectively.
Multiplying each term by the corresponding vector coefficient of $P_u$
and summing gives the Jacobi identity acting on $u$.
For each fixed coefficient in the three variables, the nonnegative total
degree of the numerator is fixed, so only finitely many terms occur.

Let $B_{a,b}$ be the operator-valued formal series obtained by subtracting
the right-hand side of the Jacobi identity from the left-hand side.
The common kernel of these series means the common kernel of all their
coefficient operators.
The mode commutator formula and the Jacobi identity in $V$ give
\begin{equation}
\label{eq:jacobi-propagation}
 [c(s),B_{a,b}]
 =\sum_{i\geq0}\binom si z_1^{s-i}B_{c_i a,b}
 +\sum_{i\geq0}\binom si z_2^{s-i}B_{a,c_i b}.
\end{equation}
More explicitly, in the iterate term use
\[
 c_jY(a,z_0)b
 =Y(a,z_0)c_jb+\sum_{i=0}^j\binom ji z_0^{j-i}Y(c_i a,z_0)b,
\]
and apply the integer identity
\[
 \begin{aligned}
 \sum_{j\geq i}\binom sj\binom ji z_2^{s-j}z_0^{j-i}
 &=\binom si\sum_{h\geq0}\binom{s-i}{h}z_2^{s-i-h}z_0^h\\
 &=\binom si\iota_{z_2,z_0}(z_2+z_0)^{s-i}.
 \end{aligned}
\]
For negative integers $s$, the right-hand side is expanded in $z_0/z_2$.
When computing the commutator with a mode in the iterate term, for each
fixed coefficient $a_rb$ of $z_0$, truncation gives $c_j(a_rb)=0$ for all
sufficiently large $j$.
If $Y(c_i a,z_0)b\in z_0^{-N_i}V[[z_0]]$, fixing the coefficient of
$z_0^q$ bounds the index in the new expansion by $0,\ldots,q+N_i$.
The delta factor in the iterate term has only nonnegative powers of
$z_0$, so after multiplication by that factor the same bound applies
to the sum of the two expansion indices.
Thus each fixed coefficient involves only finitely many terms, and the
delta identity allows $(z_2+z_0)^{s-i}$ to be replaced by $z_1^{s-i}$.
The two final sums over $i$ are finite by truncation of $c_i a$ and
$c_i b$. Hence \eqref{eq:jacobi-propagation} holds for every integer $s$.

The common kernel of all $B_{a,b}$ contains $U$.
If $v$ belongs to this common kernel, then for each mode $c(s)$,
equation~\eqref{eq:jacobi-propagation} gives
\[
 \begin{aligned}
 B_{a,b}c(s)v
 &=c(s)B_{a,b}v-[c(s),B_{a,b}]v\\
 &=-\sum_{i\geq0}\binom si z_1^{s-i}B_{c_i a,b}v
   -\sum_{i\geq0}\binom si z_2^{s-i}B_{a,c_i b}v=0.
 \end{aligned}
\]
Thus the common kernel contains $U$ and is stable under every mode.
Since $M_n(U)=U(\widehat V)U$, it equals $M_n(U)$.
Together with the vacuum property, truncation, and the nonnegative
integer grading, this proves that $M_n(U)$ is admissible.
\end{proof}

Using the canonical isomorphism in Theorem~\ref{thm:modified-induction},
we henceforth regard $U$ as the degree-$n$ subspace of $M_n(U)$.
Under this isomorphism, the zero-mode action agrees with the given action
on $U$. Thus
\begin{equation}
\label{eq:standard-induction-generation}
 M_n(U)(n)=U,\qquad M_n(U)=U(\widehat V)U.
\end{equation}

For the induced module constructed from the enveloping algebra, Griffin
\citep[Construction~4.6 and Lemma~4.7]{Griffin26} takes the quotient
by its largest graded submodule with zero degree-$n$ subspace and proves
that every nonzero graded submodule of the quotient has nonzero
intersection with the degree-$n$ subspace.
For the module $M_n(U)$ constructed here, we describe the corresponding
submodule using projection onto degree $n$. This description will be
used below to verify that the induced homomorphisms are well-defined on
the quotients.
Let $p_n:M_n(U)\to M_n(U)(n)=U$ be the projection associated with
the given direct sum grading, and let $U^*=\Hom_F(U,F)$ be the algebraic
dual. For $u^*\in U^*$ and $v\in M_n(U)$, define the pairing
$\langle u^*,v\rangle_U=u^*(p_nv)$, and set
\begin{equation}
\label{eq:pairing-submodule}
 \overline J=
 \{v\in M_n(U)\mid
 \langle u^*,xv\rangle_U=0\text{ for all }u^*\in U^*,\ x\in U(\widehat V)\}.
\end{equation}
The algebraic dual separates points, so the condition is equivalent to
$p_n(xv)=0$ for every $x\in U(\widehat V)$.

This subspace is the largest graded $V$-submodule of $M_n(U)$ whose
degree-$n$ subspace is zero. Indeed, if $v\in\overline J$, then
$\langle u^*,xyv\rangle_U=0$ for all $x,y\in U(\widehat V)$ and
$u^*\in U^*$, so $yv\in\overline J$. Write
$v=\sum_dv_d\in\overline J$. For fixed $d$ and
$x\in U(\widehat V)_{n-d}$, we have $p_n(xv)=p_n(xv_d)$.
For homogeneous $x$ of any other degree, $p_n(xv_d)=0$.
Thus each $v_d\in\overline J$, proving that this submodule is graded.
Taking the identity element $x=1$ on the degree-$n$ subspace gives
$\overline J(n)=0$. If $N$ is a graded submodule with $N(n)=0$, then
$p_n(xv)=0$ for every $v\in N$ and every $x$, so $N\subseteq\overline J$.
We therefore define the quotient module
\begin{equation}
\label{eq:standard-induced-quotient}
 L_n(U)=M_n(U)/\overline J,\qquad
 L_n(U)(n)=U,\qquad L_n(U)=U(\widehat V)U,
\end{equation}
where the equalities for the degree-$n$ subspace and generation follow
from \eqref{eq:standard-induction-generation} and $\overline J(n)=0$.

For a nonzero homogeneous vector $\xi\in L_n(U)(d)$, there exists
$\widetilde\xi\in M_n(U)(d)$ such that
$\widetilde\xi+\overline J=\xi$. Since
$\widetilde\xi\notin\overline J$, there is $x\in U(\widehat V)$ such that
$p_n(x\widetilde\xi)\neq0$.
Only the component of degree $n-d$ in the finite homogeneous
decomposition of $x$ can contribute to this projection, so we may take
$x\in U(\widehat V)_{n-d}$. Then
$0\neq x\widetilde\xi=p_n(x\widetilde\xi)\in U$, and $x\xi\neq0$
because $\overline J(n)=0$. Consequently,
\begin{equation}
\label{eq:induced-nonzero}
 0\neq\xi\in L_n(U)(d)
 \quad\Longrightarrow\quad
 \text{there exists $x\in U(\widehat V)_{n-d}$ such that $0\neq x\xi\in U$}.
\end{equation}
Choosing a nonzero homogeneous vector in any nonzero graded submodule
$N\subseteq L_n(U)$ and applying \eqref{eq:induced-nonzero} gives
$N(n)\neq0$.

\section{The structure of induced modules}
\label{sec:induced-structure}

\subsection{The kernel of the canonical map and the graded structure of induced modules}

For $n\geq1$, set $I_n=\widetilde O_{n-1}(V)/\widetilde O_n(V)$,
the kernel of the surjection between consecutive algebras in
Proposition~\ref{prop:modified-zhu-algebra}. A $\widetilde A_n(V)$-module
is said to factor through $\widetilde A_{n-1}(V)$ when $I_n$ acts
trivially on it.
For any unital $\widetilde A_n(V)$-module $U$, define
$\Ann_U(I_n)=\{u\in U:I_nu=0\}$. Since $I_n$ is a two-sided ideal,
$I_nu=0$ implies $I_n(au)=(I_na)u=0$, so $\Ann_U(I_n)$ is a submodule.
It contains every submodule annihilated by $I_n$ and is therefore the
largest submodule that factors through the algebra at the preceding
level. In particular, every nonzero submodule of $U$ does not factor
through that algebra if and only if $\Ann_U(I_n)=0$.
All modules over associative algebras in this paper are unital.

For a weak $V$-module $M$ and $w\in M$, the weak submodule generated by
$w$ is $\operatorname{span}_F\{a_jw\mid a\in V,\ j\in\mathbb Z\}$
\citep[Lemma~2.1]{Ren17}. Suppose that $M$ is admissible and $w\in M(d)$ is
homogeneous. For any $q\in\mathbb Z$, decomposing each $a$ into its
finitely many homogeneous components and then taking the component of
degree $d+q$ gives
\begin{equation}
\label{eq:single-mode-degree}
 U(\widehat V)_q w
 =\operatorname{span}_F
 \{a_jw\mid a\in V\text{ homogeneous},\ j\in\mathbb Z,\ \wt a-j-1=q\}.
\end{equation}
The left-hand side is the linear span of all $xw$ with
$x\in U(\widehat V)_q$. Enveloping algebra elements of degree $q$ send
$w$ into $M(d+q)$, and every vector on the right is obtained by applying
a single mode of degree $q$. This proves the equality of the two spaces.

For any unital $\widetilde A_n(V)$-module $U$, put $L=L_n(U)$.
Using the canonical $\widetilde A_n(V)$-module isomorphism $U\cong L(n)$, define
\[
 \iota_U:U\longrightarrow\Omega_n(L)/\Omega_{n-1}(L),\qquad
 u\longmapsto u+\Omega_{n-1}(L).
\]
The nonnegative grading gives $U\subseteq\Omega_n(L)$.
The space $\Omega_{n-1}(L)$ is stable under zero modes. Pulling back its
action along the natural surjection
$\widetilde A_n(V)\to\widetilde A_{n-1}(V)$ makes it a
$\widetilde A_n(V)$-module. This action is still given by zero modes,
so it agrees with the restriction of the action on $\Omega_n(L)$.
The quotient above therefore has a well-defined
$\widetilde A_n(V)$-module structure. Since $\iota_U$ commutes with
every zero mode, it is a $\widetilde A_n(V)$-module homomorphism.
We now determine its kernel.

\begin{lemma}
\label{lem:canonical-kernel}
Let $n\geq1$, let $U$ be a nonzero unital $\widetilde A_n(V)$-module,
and put $L=L_n(U)$. Then
\begin{equation}
\label{eq:canonical-kernel-annihilator}
 \ker\iota_U
 =U\cap\Omega_{n-1}(L)
 =\Ann_U(I_n).
\end{equation}
Consequently, $\iota_U$ is injective if and only if every nonzero
$\widetilde A_n(V)$-submodule of $U$ does not factor through
$\widetilde A_{n-1}(V)$.
\end{lemma}

\begin{proof}
By the definition of $\iota_U$,
$\ker\iota_U=\{u\in U:u+\Omega_{n-1}(L)=0\}=U\cap\Omega_{n-1}(L)$.
For $u\in U\cap\Omega_{n-1}(L)$ and $c\in\widetilde O_{n-1}(V)$,
the zero-mode action gives $(c+\widetilde O_n(V))u=o(c)u=0$, so $I_nu=0$.
Hence $U\cap\Omega_{n-1}(L)\subseteq\Ann_U(I_n)$.

Conversely, let $u\in\Ann_U(I_n)$. For homogeneous $a,b\in V$, write
$\alpha=\wt a$ and $\beta=\wt b$. Take $m=n-1$, $t=1$, and $s=0$
in \eqref{eq:zero-mode-formula}, so that $N=2n+1$.
The nonnegative grading gives
\[
 \begin{aligned}
 b_{\beta+n+\ell}L(n)
 &\subseteq L\bigl(n+\beta-(\beta+n+\ell)-1\bigr)
 =L(-\ell-1)=0 &&(\ell\geq0),\\
 a_{\alpha+n-1+\ell}L(n)
 &\subseteq L\bigl(n+\alpha-(\alpha+n-1+\ell)-1\bigr)
 =L(-\ell)=0 &&(\ell\geq1).
 \end{aligned}
\]
Restricting \eqref{eq:zero-mode-formula} to $L(n)$ and using
$-(-1)^{2n+1}\binom{-2n-1}{0}=1$, we obtain
\begin{equation}
\label{eq:isolated-bottom-product}
 o(a\circ_{n-1,1}^{0}b)\big|_{L(n)}
 =b_{\beta-n-1}a_{\alpha+n-1}\big|_{L(n)}.
\end{equation}
Since $a\circ_{n-1,1}^{0}b\in\widetilde O_{n-1}(V)$ and $I_nu=0$,
equation~\eqref{eq:isolated-bottom-product} gives
$b_{\beta-n-1}a_{\alpha+n-1}u=o(a\circ_{n-1,1}^{0}b)u=0$.
Set $v=a_{\alpha+n-1}u\in L(0)$. Equation~\eqref{eq:single-mode-degree} gives
\[
 U(\widehat V)_nv
 =\operatorname{span}_F\{b_{\wt b-n-1}v:b\in V\text{ homogeneous}\}
 =0.
\]
Equation~\eqref{eq:induced-nonzero} therefore gives $a_{\alpha+n-1}u=v=0$.
For $j\geq\alpha+n$, the nonnegative grading also gives
$a_ju\in L(n+\alpha-j-1)=0$. It follows that $a_ju=0$ for all
$j\geq\wt a+n-1$, so $u\in\Omega_{n-1}(L)$.
This proves $\Ann_U(I_n)\subseteq U\cap\Omega_{n-1}(L)$ and hence
$\ker\iota_U=U\cap\Omega_{n-1}(L)=\Ann_U(I_n)$.

Finally, $\ker\iota_U=\Ann_U(I_n)$ shows that $\iota_U$ is injective
if and only if $\Ann_U(I_n)=0$. By the maximal property of
$\Ann_U(I_n)$ established above, this is equivalent to the condition
that every nonzero $\widetilde A_n(V)$-submodule of $U$ does not
factor through $\widetilde A_{n-1}(V)$.
\end{proof}

The kernel formula in Lemma~\ref{lem:canonical-kernel} and the zero-mode
action formula in Theorem~\ref{thm:zero-mode-action} give the following
equivalent conditions, relating submodules of $U$, graded submodules of
$L_n(U)$, and the subspaces $\Omega_r(L_n(U))$.

\begin{theorem}
\label{thm:omega-quotient}
Let $n\geq1$, let $U$ be a nonzero unital $\widetilde A_n(V)$-module,
and put $L=L_n(U)$. The following conditions are equivalent:
\begin{enumerate}
 \item Every nonzero $\widetilde A_n(V)$-submodule of $U$ does not factor through $\widetilde A_{n-1}(V)$.
 \item Every nonzero graded $V$-submodule $N\subseteq L$ satisfies $N(0)\neq0$.
 \item For every integer $r\geq0$, there is an equality of graded subspaces
 \begin{equation}
 \label{eq:all-omega-layers}
  \Omega_r(L)=\bigoplus_{d=0}^{r}L(d);
 \end{equation}
 \item The canonical map $\iota_U:U\to\Omega_n(L)/\Omega_{n-1}(L)$ is injective
 (equivalently, it is a $\widetilde A_n(V)$-module isomorphism).
 \item $\Omega_0(L)=L(0)$.
\end{enumerate}
\end{theorem}

\begin{proof}
\textup{(1)$\Rightarrow$(2)}.
Let $N\subseteq L$ be a nonzero graded submodule.
Equation~\eqref{eq:induced-nonzero} gives $N(n)\neq0$. Zero modes
preserve both $N$ and the given grading, so $N(n)$ is a nonzero
$\widetilde A_n(V)$-submodule of $U$. If $N(0)=0$, then, for all $k\geq n$,
\[
 \begin{aligned}
 U(\widehat V)_{-k}N(n)&\subseteq N(n-k)=0,\\
 \varphi_L\bigl(U(\widehat V)_k\bigr)
 \varphi_L\bigl(U(\widehat V)_{-k}\bigr)N(n)&=0.
 \end{aligned}
\]
Apply Theorem~\ref{thm:zero-mode-action} with $m=n-1$.
For each $c\in\widetilde O_{n-1}(V)$ and $u\in N(n)$,
equation~\eqref{eq:MF} expresses $o(c)u$ as $\sum_{k\geq n}T_ku$,
where $T_k\in\varphi_L(U(\widehat V)_k)\varphi_L(U(\widehat V)_{-k})$
and only finitely many terms can be nonzero for fixed $u$.
Thus $o(c)u=\sum_{k\geq n}T_ku=0$, and hence
$I_nN(n)=o(\widetilde O_{n-1}(V))N(n)=0$.
It follows that $N(n)$ factors through $\widetilde A_{n-1}(V)$,
contrary to condition~\textup{(1)}.

\textup{(2)$\Rightarrow$(3)}.
For $0\leq d\leq r$ and $\wt a-j-1<-r$, the nonnegative grading gives
$a_jL(d)\subseteq L(d+\wt a-j-1)=0$. Hence
$\bigoplus_{d=0}^{r}L(d)\subseteq\Omega_r(L)$.
Suppose that $0\neq\xi\in\Omega_r(L)\cap L(d)$ with $d>r$.
The submodule $N=U(\widehat V)\xi$ generated by $\xi$ is nonzero and
graded. Since $-d<-r$, equation~\eqref{eq:single-mode-degree} and the
definition of $\Omega_r(L)$ give
\[
 N(0)=U(\widehat V)_{-d}\xi
 =\operatorname{span}_F
 \{a_{\wt a+d-1}\xi\mid a\in V\text{ homogeneous}\}=0.
\]
This contradicts condition~\textup{(2)}. Since $\Omega_r(L)$ is a
graded subspace, equation~\eqref{eq:all-omega-layers} follows.

\textup{(3)$\Rightarrow$(4)}.
Taking $r=n$ and $r=n-1$, respectively, gives
\[
 \frac{\Omega_n(L)}{\Omega_{n-1}(L)}
 =\frac{\bigoplus_{d=0}^{n}L(d)}{\bigoplus_{d=0}^{n-1}L(d)}
 \cong L(n)=U.
\]
The inverse of this isomorphism is $u\mapsto u+\Omega_{n-1}(L)$,
namely $\iota_U$. Since $\iota_U$ is a $\widetilde A_n(V)$-module
homomorphism, it is a module isomorphism and, in particular, is injective.

\textup{(4)$\Rightarrow$(1)}.
If $\iota_U$ is injective, then $\ker\iota_U=0$.
Equation~\eqref{eq:canonical-kernel-annihilator} gives $\Ann_U(I_n)=0$,
and condition~\textup{(1)} follows.

\textup{(3)$\Rightarrow$(5)$\Rightarrow$(2)}.
Taking $r=0$ in condition~\textup{(3)} gives condition~\textup{(5)}.
Conversely, suppose that condition~\textup{(5)} holds. If a nonzero
graded submodule $N$ satisfies $N(0)=0$, choose the least $d>0$ with
$N(d)\neq0$. For $\wt a-j-1<0$, the minimality of $d$ and the
nonnegative grading give $a_jN(d)\subseteq N(d+\wt a-j-1)=0$.
Condition~\textup{(5)} therefore gives
\[
 0\neq N(d)\subseteq\Omega_0(L)\cap L(d)=L(0)\cap L(d)=0,
\]
a contradiction. Thus condition~\textup{(2)} holds.
\end{proof}

\subsection{Homomorphisms of induced modules and nonsplit extensions}

For the induction functor over the complex numbers, see Dong, Li, and
Mason \citep{DLM98}. Using the induction construction above and the
pairing in \eqref{eq:pairing-submodule}, we now verify naturality of the
induced modules with respect to homomorphisms over the present field.

For admissible graded $V$-modules $M,N$, let
$\Hom_V^{\mathrm{gr}}(M,N)$ denote the $F$-vector space of
grading-preserving $V$-module homomorphisms $f:M\to N$, that is,
homomorphisms satisfying $f(M(d))\subseteq N(d)$ for every integer $d\geq0$.

\begin{proposition}\label{prop:induction-naturality}
Let $n\geq1$, and let $U,U'$ be unital $\widetilde A_n(V)$-modules.
Restriction to degree $n$ gives a canonical $F$-linear isomorphism
\[
 \Hom_V^{\mathrm{gr}}(L_n(U),L_n(U'))
 \xrightarrow{\sim}\Hom_{\widetilde A_n(V)}(U,U').
\]
Write $f\mapsto L_n(f)$ for its inverse. This assignment respects
identity homomorphisms and composition, and it preserves injective and
surjective homomorphisms. For every $u\in U$, the map induced by $L_n(f)$
from $\Omega_n(L_n(U))/\Omega_{n-1}(L_n(U))$ to
$\Omega_n(L_n(U'))/\Omega_{n-1}(L_n(U'))$ sends $\iota_U(u)$ to
$\iota_{U'}(f(u))$.
\end{proposition}

\begin{proof}
Let $f:U\to U'$ be a $\widetilde A_n(V)$-module homomorphism.
The degree-zero part of $P_n$ acts through $\widetilde A_n(V)$, and its
negative-degree part acts trivially on both modules. Hence $f$ is a
$P_n$-module homomorphism and, by the universal property of the
enveloping algebra, a $U(P_n)$-module homomorphism. Thus, for
$x\in U(\widehat V)$, $p\in U(P_n)$, and $u\in U$,
we have $xp\otimes f(u)=x\otimes pf(u)=x\otimes f(pu)$.
The balancing relation therefore gives a well-defined graded
$U(\widehat V)$-module homomorphism
\[
 \overline M_n(f):\overline M_n(U)\longrightarrow\overline M_n(U'),
 \qquad x\otimes u\longmapsto x\otimes f(u).
\]
Let $W'$ denote the relation subspace associated with $U'$.
For the coefficients defining $W$, we have
\[
 \begin{aligned}
 \overline M_n(f)w_{s,t}(a,b;u)&=w_{s,t}(a,b;f(u))\in W',\\
 \overline M_n(f)\bigl(U(\widehat V)W\bigr)
 &=U(\widehat V)\overline M_n(f)(W)\subseteq U(\widehat V)W'.
 \end{aligned}
\]
It therefore descends to a graded $V$-module homomorphism
$M_n(f):M_n(U)\to M_n(U')$. Let $p_n':M_n(U')\to U'$ denote the
projection onto degree $n$. For homogeneous $v\in M_n(U)(d)$, the
grading and the equality $M_n(f)|_U=f$ give
\[
 p_n'M_n(f)v=
 \begin{cases}f(v),&d=n,\\0,&d\neq n,\end{cases}
 =f(p_nv).
\]
By linearity, $p_n'M_n(f)=fp_n$.
Let $f^*:U'^*\to U^*$ be the dual map, given by
$f^*(u'^*)=u'^*\circ f$. For $v\in M_n(U)$ and $u'^*\in U'^*$,
the definition of the pairing gives
\begin{equation}\label{eq:pairing-naturality}
 \begin{aligned}
 \langle u'^*,M_n(f)v\rangle_{U'}
 &=u'^*\bigl(p_n'M_n(f)v\bigr)
 =u'^*\bigl(f(p_nv)\bigr)\\
 &=(f^*u'^*)(p_nv)=\langle f^*u'^*,v\rangle_U.
 \end{aligned}
\end{equation}
Write $\overline J'$ for the submodule defined for $U'$ by
\eqref{eq:pairing-submodule}. If $v\in\overline J$, then
\eqref{eq:pairing-naturality} and the $U(\widehat V)$-linearity of
$M_n(f)$ give, for all $x\in U(\widehat V)$ and $u'^*\in U'^*$,
\[
 \begin{aligned}
 \langle u'^*,xM_n(f)v\rangle_{U'}
 &=\langle u'^*,M_n(f)(xv)\rangle_{U'}\\
 &=\langle f^*u'^*,xv\rangle_U=0.
 \end{aligned}
\]
Thus $M_n(f)(\overline J)\subseteq\overline J'$, so we obtain a graded
homomorphism $L_n(f):L_n(U)\to L_n(U')$ on the quotients in
\eqref{eq:standard-induced-quotient}, explicitly given by
$L_n(f)(v+\overline J)=M_n(f)v+\overline J'$, with $L_n(f)|_U=f$.

For any graded $V$-module homomorphism $h:L_n(U)\to L_n(U')$,
homogeneous $a\in V$, and $u\in U$,
\[
 h\bigl((a+\widetilde O_n(V))u\bigr)
 =h(o(a)u)=o(a)h(u)=(a+\widetilde O_n(V))h(u).
\]
Hence $h|_U$ is $\widetilde A_n(V)$-linear.
If $h_1|_U=h_2|_U$, then
$h_1(xu)=xh_1(u)=xh_2(u)=h_2(xu)$ for $x\in U(\widehat V)$ and $u\in U$.
Since $L_n(U)=U(\widehat V)U$, the extension is unique.
The restriction map is $F$-linear, and the existence and uniqueness just
proved show that it is a linear isomorphism. Identity maps and
composites extend their respective restrictions to degree $n$.
Uniqueness therefore gives
$L_n(\mathrm{id}_U)=\mathrm{id}_{L_n(U)}$ and
$L_n(gf)=L_n(g)L_n(f)$, where $g:U'\to U''$ and $U''$ satisfies the
same assumptions.

If $f$ is injective, then $(\ker L_n(f))(n)=\ker f=0$, and
\eqref{eq:induced-nonzero} gives $\ker L_n(f)=0$.
If $f$ is surjective, then
\[
 L_n(U')=U(\widehat V)U'
 =U(\widehat V)f(U)\subseteq\operatorname{im}L_n(f),
\]
so $L_n(f)$ is surjective.
Finally, for $w\in\Omega_r(L_n(U))$ and $\wt a-j-1<-r$, we have
$a_jL_n(f)w=L_n(f)(a_jw)=0$.
Thus $L_n(f)(\Omega_r(L_n(U)))\subseteq\Omega_r(L_n(U'))$, and
$L_n(f)$ induces the well-defined map
\[
 \begin{aligned}
 \Omega_n(L_n(U))/\Omega_{n-1}(L_n(U))
 &\longrightarrow\Omega_n(L_n(U'))/\Omega_{n-1}(L_n(U')),\\
 w+\Omega_{n-1}(L_n(U))
 &\longmapsto L_n(f)w+\Omega_{n-1}(L_n(U')).
 \end{aligned}
\]
For $u\in U$, this quotient map sends $\iota_U(u)$ to
\[
 L_n(f)u+\Omega_{n-1}(L_n(U'))
 =f(u)+\Omega_{n-1}(L_n(U'))=\iota_{U'}(f(u)).
\]
This proves the asserted compatibility. When every nonzero submodule
of $U$ and every nonzero submodule of $U'$ do not factor through the
algebra at the preceding level, Theorem~\ref{thm:omega-quotient} further
shows that both canonical maps are isomorphisms.
\end{proof}

When every nonzero submodule of $U$ does not factor through the algebra
at the preceding level, Theorem~\ref{thm:omega-quotient} gives a canonical
module isomorphism between $U=L_n(U)(n)$ and
$\Omega_n(L_n(U))/\Omega_{n-1}(L_n(U))$.
Over the complex numbers, Barron, Vander Werf, and Yang
\citep[Proposition~3.9]{BVWY19} prove that if an $A_n(V)$-module $U$ is
indecomposable and does not factor through $A_{n-1}(V)$, then the
standard induced module $L_n(U)$ is graded indecomposable.
Griffin \citep[Lemma~4.8]{Griffin26} proves that if $\mathsf M$ is an
indecomposable $\mathsf A_d(V)$-module, then the induced module
$\Phi_d^{\mathsf L}(\mathsf M)$ constructed from the enveloping algebra
is indecomposable in the category of graded modules.
For the quotient $L_n(U)$ defined here, we use the argument based on
graded generation in the proof of that lemma for the forward implication,
and naturality of induced homomorphisms for the converse.

\begin{corollary}
\label{cor:indecomposable}
Let $n\geq1$, and let $U$ be a nonzero unital $\widetilde A_n(V)$-module.
Then $U$ is indecomposable if and only if $L_n(U)$ is indecomposable in
the category of admissible graded $V$-modules.
\end{corollary}

\begin{proof}
Suppose that $U$ is indecomposable and that $L_n(U)=W_1\oplus W_2$ is
a graded direct sum decomposition. Zero modes preserve each graded
submodule, so $U=W_1(n)\oplus W_2(n)$ is a direct sum decomposition of
$\widetilde A_n(V)$-modules. Since $U$ is indecomposable, we may assume
$W_1(n)=0$. Let $\pi_1:L_n(U)\to W_1$ be the graded projection.
The generation property in \eqref{eq:standard-induced-quotient} gives
\[
 W_1=\pi_1\bigl(U(\widehat V)U\bigr)
 =U(\widehat V)\pi_1(U)=0.
\]
Thus $L_n(U)$ is graded indecomposable.

Conversely, suppose that $U=U_1\oplus U_2$ with both summands nonzero.
Let $e$ be the idempotent $\widetilde A_n(V)$-module endomorphism
given by projection onto $U_1$.
Proposition~\ref{prop:induction-naturality} gives a graded endomorphism
$L_n(e)$ satisfying $L_n(e)^2=L_n(e^2)=L_n(e)$ and $L_n(e)|_U=e$.
Hence
\[
 \begin{aligned}
 L_n(U)&=\operatorname{im}L_n(e)\oplus\ker L_n(e),\\
 (\operatorname{im}L_n(e))(n)&=\operatorname{im}e=U_1\neq0,\\
 (\ker L_n(e))(n)&=\ker e=U_2\neq0.
 \end{aligned}
\]
This is a graded direct sum decomposition with two nonzero summands,
which proves the converse.
\end{proof}

The following theorem uses induced homomorphisms to construct nonsplit
short exact sequences of admissible graded $V$-modules from nonsplit
extensions satisfying the stated simplicity assumptions. The assumption
on the quotient $T$ is that it is a simple $\widetilde A_n(V)$-module.
If $T$ also factors through $\widetilde A_{n-1}(V)$, the theorem applies
in particular to nonsplit extensions whose submodule $S$ does not factor
through $\widetilde A_{n-1}(V)$ and whose quotient $T$ does.

\begin{theorem}\label{thm:nonsplit-extension}
Let $n\geq1$, and let
\[
 0\longrightarrow S\xrightarrow{\iota}E
 \xrightarrow{\pi}T\longrightarrow0
\]
be a nonsplit short exact sequence of $\widetilde A_n(V)$-modules,
with $S$ and $T$ simple. Suppose that $S$ does not factor through
$\widetilde A_{n-1}(V)$. Then $E$ is indecomposable, and every nonzero
$\widetilde A_n(V)$-submodule of $E$ does not factor through
$\widetilde A_{n-1}(V)$. Moreover, $L_n(E)$ is graded indecomposable.
The natural homomorphism $L_n(\iota):L_n(S)\to L_n(E)$ is injective.
Set $Q_E=L_n(E)/L_n(\iota)(L_n(S))$. Its homogeneous subspace of
degree $n$ satisfies $Q_E(n)\cong T$ as a $\widetilde A_n(V)$-module;
in particular, $Q_E\neq0$. Furthermore, the short exact sequence
\[
 0\longrightarrow L_n(S)
 \xrightarrow{L_n(\iota)}
 L_n(E)\longrightarrow Q_E\longrightarrow0
\]
is nonsplit in the category of admissible graded $V$-modules and remains
a nonsplit short exact sequence of weak $V$-modules after forgetting
the grading. The quotient $Q_E$ is generated by its homogeneous
subspace of degree $n$.
\end{theorem}

\begin{proof}
Set $S_E=\iota(S)$. If a nonzero submodule $N\subseteq E$ satisfies
$N\cap S_E=0$, then simplicity of $T$ gives
$\ker(\pi|_N)=N\cap S_E=0$ and $0\neq\pi(N)=T$.
Hence $\pi|_N:N\to T$ is an isomorphism. Regarding its inverse as a
homomorphism from $T$ to $E$, we obtain
$\pi\circ(\pi|_N)^{-1}=(\pi|_N)(\pi|_N)^{-1}=\mathrm{id}_T$,
contrary to the nonsplitting of the original short exact sequence.
Simplicity of $S_E$ therefore gives $0\neq N\cap S_E=S_E$, so
$S_E\subseteq N$. If $E=E_1\oplus E_2$ with both summands nonzero,
then $0\neq S_E\subseteq E_1\cap E_2=0$, a contradiction.
Thus $E$ is indecomposable.

By hypothesis, $I_nS\ne0$. Since $I_nS$ is a submodule of the simple
module $S$, we have $I_nS=S$. For every nonzero submodule $N\subseteq E$,
it follows that
\[
 I_nN\supseteq I_nS_E=I_n\iota(S)=\iota(I_nS)=S_E\neq0.
\]
This proves that every nonzero submodule of $E$ does not factor through
$\widetilde A_{n-1}(V)$. Conversely, if $S$ factors through that
algebra, then $S_E$ itself is such a nonzero submodule. Thus, under the
present assumption of a nonsplit extension of simple modules, the
submodule condition is equivalent to $S$ not factoring through
$\widetilde A_{n-1}(V)$.

The indecomposability of $E$ and Corollary~\ref{cor:indecomposable}
give the graded indecomposability of $L_n(E)$.

Proposition~\ref{prop:induction-naturality} gives the graded homomorphism
$L_n(\iota)$, whose restriction to degree $n$ is the injective map
$\iota$. Thus $(\ker L_n(\iota))(n)=\ker\iota=0$.
Equation~\eqref{eq:induced-nonzero} gives $\ker L_n(\iota)=0$, so
$L_n(\iota)$ is injective.

Define $Q_E=L_n(E)/L_n(\iota)(L_n(S))$.
The induced homomorphism preserves the grading, so its image is a
graded $V$-submodule and $Q_E$ is an admissible graded $V$-module.
Let $q:L_n(E)\to Q_E$ be the quotient map.
The grading-preserving property of $L_n(\iota)$ and the canonical
$\widetilde A_n(V)$-module isomorphism between $L_n(E)(n)$ and $E$ give
\[
 \begin{aligned}
 \bigl(\operatorname{im}L_n(\iota)\bigr)(n)
 &=L_n(\iota)\bigl(L_n(S)(n)\bigr)=\iota(S),\\
 Q_E(n)
 &=\frac{L_n(E)(n)}
 {\bigl(\operatorname{im}L_n(\iota)\bigr)(n)}
 =\frac{E}{\iota(S)}\cong T.
 \end{aligned}
\]
The map $\pi:E\to T$ induces a $\widetilde A_n(V)$-module isomorphism
$E/\iota(S)\cong T$, giving a $\widetilde A_n(V)$-module isomorphism
$\theta:Q_E(n)\to T$ with $\theta q|_{E}=\pi$.
In particular, $Q_E\neq0$.

If the resulting short exact sequence splits in the graded category,
choose a graded $V$-module section $\sigma:Q_E\to L_n(E)$ such that
$q\sigma=\mathrm{id}_{Q_E}$. In degree $n$, the map
$s=(\sigma|_{Q_E(n)})\theta^{-1}:T\to E$ is a
$\widetilde A_n(V)$-module homomorphism, and
\[
 \pi s
 =\theta(q\sigma)|_{Q_E(n)}\theta^{-1}
 =\theta\,\mathrm{id}_{Q_E(n)}\theta^{-1}=\mathrm{id}_T.
\]
Thus $s$ splits the original short exact sequence, contrary to the
hypothesis. The induced short exact sequence is therefore nonsplit in
the category of graded modules. Since $q(E)=Q_E(n)$, generation gives
\[
 Q_E=q(L_n(E))=q\bigl(U(\widehat V)E\bigr)
 =U(\widehat V)q(E)=U(\widehat V)Q_E(n).
\]

Now suppose that there exists a weak module section $\sigma:Q_E\to L_n(E)$.
For $v\in Q_E(d)$, let $\sigma_0(v)$ be the component of $\sigma(v)$
of degree $d$, and extend linearly to $Q_E$. This defines a graded
linear map. Write $p_d^Q$ and $p_d^L$ for the degree projections on the
two modules, with projections onto negative degrees defined to be zero.
If a homogeneous mode $a_r$ has degree $k$, then
\[
 \begin{aligned}
 \sigma_0(a_rv)
 &=p_{d+k}^L\sigma(a_rv)
 =p_{d+k}^La_r\sigma(v)
 =a_rp_d^L\sigma(v)
 =a_r\sigma_0(v),\\
 q\sigma_0(v)&=qp_d^L\sigma(v)
 =p_d^Qq\sigma(v)=v.
 \end{aligned}
\]
Thus $\sigma_0$ is a graded section, contrary to the nonsplitting
proved above. Hence the short exact sequence remains nonsplit as a
sequence of weak modules after forgetting the grading.
\end{proof}

\end{document}